\documentclass[11pt]{amsart}
\usepackage{mathabx}
\usepackage[T1]{fontenc} 
\usepackage{amsaddr}
\usepackage{amssymb}
\usepackage{amsmath}
\usepackage{amsfonts}
\usepackage{amssymb}
\usepackage{amsthm}
\usepackage{float}
\usepackage{xspace}
\usepackage{epsfig}
\usepackage{xypic}
 \usepackage{hyperref}
 \usepackage{color}

\hypersetup{urlcolor=blue, citecolor=red}

\usepackage{color}
\definecolor{Green}{rgb}{0,.8,0}

\usepackage{graphicx}              

\usepackage{microtype} 

\usepackage[hmarginratio=1:1,top=32mm,columnsep=20pt]{geometry} 
\usepackage{multicol} 

\usepackage[hang, small,labelfont=bf,up,textfont=it,up]{caption} 
\usepackage{booktabs} 
\usepackage{float} 

\usepackage{lettrine} 
\usepackage{paralist} 

\newtheorem{proposition}{Proposition}[section]
\newtheorem{theorem}[proposition]{Theorem}
\newtheorem{lemma}[proposition]{Lemma}
\newtheorem{remark}[proposition]{Remark}
\newtheorem{definition}[proposition]{Definition}
\newtheorem{corollary}[proposition]{Corollary}

\makeatletter

\@addtoreset{equation}{section}
\makeatother

\def\CC{\overline{\mathcal{C}}}   

\def\CM{\mathcal{M}}

\def\id{{\mathrm{Id}}}

\def\sym{\mathtt{sym}}
\def\Sym{\mathrm{Sym}}

\newcommand{\hlie}{\mathfrak{h}}

\newcommand{\adiag}{A_{\mathrm{diag}}}
\newcommand{\asub}{A_{\mathrm{sub}}}

\begin{document}
\sloppy
\title[Effective Morales-Ramis-Sim\'o theorem]{Liouville integrability: an effective Morales-Ramis-Sim\'o theorem}
\author{A. Aparicio-Monforte}
\address{Dawson College, \\ Westmount, Montreal QC, Canada}
\email{aamonforte@dawsoncollege.qc.ca}
\author{T. Dreyfus}
\address{Universit\'e Paul Sabatier - Institut de Math\'ematiques de Toulouse, \\ 118 route de Narbonne, 31062 Toulouse, France}
\email{tdreyfus@math.univ-toulouse.fr}
%
\author{J.-A. Weil}
\address{XLIM - Universit\'e de Limoges \\ 123 avenue Albert Thomas, 87060 Limoges Cedex, France}
\email{weil@unilim.fr}

\subjclass[2010]{}
  \renewcommand{\subjclassname}{%
    \textup{2010} Mathematics Subject Classification}


\date{\today}

\begin{abstract}
Consider a complex Hamiltonian system and an integral curve. In this paper, we give an effective and efficient  procedure to put the variational equation of any order along the integral curve in reduced form provided that the previous one is in reduced form with an abelian Lie algebra. Thus, we obtain an effective way to check the Morales-Ramis-Sim\'o criterion for testing meromorphic Liouville integrability of Hamiltonian systems.
\end{abstract} 

\maketitle 

\newpage
\tableofcontents


\section{Introduction}
Consider a Hamiltonian system of $2n$ differential equations
$$(X_H) \; : \quad 
	\left\{
		\begin{array}{ccl}	
			\dot{q}_i & = & +\frac{\partial H}{\partial p_i} \\\\
			\dot{p}_i & = & -\frac{\partial H}{\partial q_i} 
		\end{array}
	\right.
$$
A \emph{first integral} is a function of the $q_i$ and $p_i$ which is constant along the solutions of $(X_H)$. 
The system is called (meromorphically) \emph{Liouville integrable} (or \emph{completely integrable}) when it admits
$n$ (meromorphic) first integrals $F_1,\ldots, F_n$ which are functionally independent (their differentials are linearly independent) and in involution (their Poisson brackets vanish or, equivalently, the associated Hamiltonian vector fields $X_{F_i}$ commute). We refer to the reference books \cite{AbMa78a,CuBa97a,Au08a} for more on this topic; see also Section 2 for definitions.
\\

The Ziglin-Morales-Ramis theory (see \cite{MoR,Au08a} for statements and applications) provides mathematical tools to check when a system is non-integrable. This is particularly useful as Hamiltonian systems generally come as parametrized 
	families.
The non-integrability criteria allow one to discard the vast majority of values of the parameters for which the system is not integrable.  The principle is as follows. First, we find a particular solution $\Gamma$ of the system $(X_H)$ (generally from an invariant plane found from symmetries) and we compute variational equations $(\mathrm{VE}_p)$, i.e. systems of linear differential equations governing a Taylor expansion of a solution of $(X_H)$ along the particular solution $\Gamma$. The Liouville integrability of  $(X_H)$ induces  integrability conditions on the variational equations $(\mathrm{VE}_p)$, which in turn imply properties of their monodromy or differential Galois groups. Technically, the Morales-Ramis-Sim\'o theorem states that if $(X_H)$ is integrable, then the Lie algebras of the differential Galois groups of all variational equations $(\mathrm{VE}_p)$ must be abelian (all these terms are defined in Section 2).
\\

The strength of this criterion is that it turns a geometric condition (integrability) into an algebraic one (abelianity of a Lie algebra), thus paving the way for possible computations. 
However, although there exist general algorithms to compute differential Galois groups of reducible systems such as the variational equations $(\mathrm{VE}_p)$ (\cite{Fe15a,Re14a} or \cite{Ho07b}), none of them are currently even close to being practical or implemented at this time. 
Furthermore, the size of the variational equations $(\mathrm{VE}_p)$ grows fast, so only a method which uses the structure of the system to make it simpler has a chance of being efficient. 
The main goal of the present paper is to explain how to use the structure of the system to make it simpler, which will allow us to check efficiently whether its Lie algebra is abelian or not. \\\par 

Over the past decade, several approaches have been devised to concretely apply this Morales-Ramis-Sim\'o integrability criterion.\par 
For Hamiltonians of the form $H = \sum_{i=1}^{n} \frac{1}{2} p_i^2 + V(q)$, where $V$ is a potential in $q$,  the first variational equation is often a direct sum of Lam\'e equations of the form ${y''(x)=\left(n(n+1)\wp(x)+B\right)y(x)}$, where $\wp$ denotes the Weierstrass function associated to an elliptic curve. In this case, Morales has elaborated a local criterion to find obstructions to integrability on higher variational equations via local computations (see Lemmas 11 and 12  in \cite{MoRa01b} Page~79, and Proposition 7, Page~81). Maciejewski, Przybylska and Duval have elaborated techniques to handle variational equations for the case of Hamiltonians with potentials (\cite{ MaPr06a,DuMa09a,DuMa14a,DuMa15a}); 
see also the works of Combot and coauthors
\cite{Co13b,CoKo12a,BoCoSa14a}.

Another approach is to determine numerical trajectories and compute numerical monodromies around these. 
Although it is difficult to obtain rigorous proofs by these methods, they provide surprisingly precise information. 
They have been developed, for example, by Martinez and Sim\'o  \cite{MaSi09a}, by Simon and Sim\'o in the Atwood paper \cite{PPRSSW10a}, by Simon in the more recent \cite{Si14a,Si14b} and by  Salnikov \cite{Sa14a,Sa13b}.

The general strategy for turning numerical evidence into rigorous proofs is to show that a certain commutator is non-zero. This in turn yields calculations of integrals and of residues, which can be achieved algorithmically due to their $D$-finiteness. This is used by Martinez and Sim\'o  in \cite{MaSi09a} and later systematized by Combot and coauthors, see e.g. {\cite{CoKo12a,Co13b,BoCoSa14a}}.
\\

The approach that we develop in this paper follows previous work by two of the authors in \cite{ApCoWe13a,ApWe11a,ApWe12b}. We establish a \emph{reduction method}. 
Consider the $p$-th variational equation $(\mathrm{VE}_p)\; : \; Y'=A(x) Y$, where the coefficients of $A(x)$ are in a differential field 
$\mathbf{k}$. 
Given an invertible matrix $P(x)$ (a \emph{gauge transformation} matrix), performing the linear change of variable $Z=P(x)Y$ produces an  equivalent linear differential system for $Z$, denoted by $Z'=P(x)[A(x)]\; Z$. The principle of reduction methods is to look for a gauge transformation $P(x)$ such that the resulting system  $Z'=P(x)[A(x)]\; Z$ is ``as simplified as possible''. 

Let $G$ denote the differential Galois group of $(\mathrm{VE}_p)$ and $\mathfrak{g}$ be the Lie algebra of $G$. 
Following traditional works of Kolchin and Kovacic, we will say that we have a \emph{reduced form} when $P(x)[A(x)] \in \mathfrak{g}(\mathbf{k})$ (see Subsection \ref{section-reduced}). 
Despite the apparent technicality of this definition, the Kolchin-Kovacic theory shows why this is a desirable form. This is similar to the Lie-Vessiot-Guldberg theories of reduction of connections (see \cite{BlMo10a,BlMo12a} for the latter and their connections with the Kolchin-Kovacic theory of reduced forms). 
Our strategy in this paper is to compute  such a reduction matrix $P(x)$ efficiently.

After this reduction process, the Lie algebra $\mathfrak{g}$ is
easily read and its abelianity (or not) is given in the process. Furthermore, if $\mathfrak{g}$ is abelian, then this process will have prepared the system to allow an efficient reduction of the next variational equation.
\\

Our strategy can be summarized as follows.  The $p$-th variational equation $(\mathrm{VE}_p)$ is a differential system of the form $Y'=A(x)Y$ where $A(x)$ has the form
	$$ A(x) =    \left(\begin{array}{c|c}
	      A_1(x) & 0 \\\hline
	      S(x) & A_2(x) \\
	   \end{array}\right).$$
In the Morales-Ramis-Sim\'o situation (see Subsection \ref{sec25}), we may assume that the $A_i(x)$ are in reduced form and that the Lie algebra of the differential Galois group of the block diagonal system
$$ Y' = A_{\mathrm{diag}} Y, 
	   \quad \textrm{ with } \quad 
	A_{\mathrm{diag}} = \left(\begin{array}{c|c} 
	      A_1(x) & 0 \\\hline
	      0 & A_2(x) \\
	   \end{array}\right), $$
has an abelian Lie algebra. 
We show (Theorem \ref{theo1} in Subsection \ref{sec32}) that the reduction matrix may be chosen of the form
$$P(x) = 
	\left(\begin{array}{c|c} 
		\id
			 & 0 
		\\\hline 
		\sum_i f_i(x) S_i &
			\begin{array}{ccc} \id \end{array}
	\end{array}\right)
$$
where $\id$ denotes the identity matrix, where the $S_i$ are easily found from $S(x)$ and where the unknown functions $f_i(x)$ remain to be found.
In Subsection \ref{sec34}, we show how standard linear algebra allows us to find these $f_i(x)$ as rational solutions of first order linear differential equations
$y' = \lambda(x) y + \sum_i c_i b_i(x)$ where the $c_i$ are constant and where $\lambda(x)$ and the $b_i(x)$ are in a convenient field. 
\\

\noindent\textbf{Structure of the paper. }  
In Section 2, we recall the necessary notions of Liouville integrability of Hamiltonian systems, differential Galois groups, reduced forms of linear differential systems and the Morales-Ramis-Sim\'o integrability condition. This section contains only previously known material.
In Section 3, we solve a problem that is interesting in its own right : given a block triangular differential system whose diagonal blocks are in reduced form and have an abelian Lie algebra, we give a practical procedure to put the system into reduced form (and hence compute its differential Galois group). 
In Section 4, we show how to reduce the Morales-Ramis-Sim\'o condition to the latter problem and thereby provide an effective version of the Morales-Ramis-Sim\'o integrability criterion. 
In Section 5, we demonstrate the efficiency of the method by computing the reduced form and
the differential Galois groups of the first three variational equations on a four dimensional
example, originally considered in \cite{CaDuMaPr10a}.
\\

\noindent\textbf{Acknowledgments. }  
	We would like to thank G.~Casale, T.~Combot, A.~Maciejewski, J.-J.~Morales, M.~Przybylska, J.-P.~Ramis  and M.F.~Singer
	for inspiring conversations regarding the material elaborated here. This paper was initiated at an EMS conference in Bedlewo and significantly improved in Wuhan where T.~Dreyfus and J.-A.~Weil were invited by the Chinese Academy of Science and the ANR project q-diff.
	T.~Dreyfus is supported by the Labex CIMI in Toulouse.

\pagebreak[3]
\section{The Morales-Ramis-Sim\'o Integrability Condition}  \label{sec2}
\subsection{Hamiltonian Systems and Liouville Integrability}\label{sec21}

Let $(M\,,\,\omega)$ be a complex analytic symplectic manifold of complex dimension $2n$ with $n\in\mathbb{N}^{*}$. Since $M$ is locally isomorphic to an open connected domain $U\subset\mathbb{C}^{2n}$, Darboux's theorem allows us to choose a set of local coordinates  $(q\,,\,p)=(q_1 \,\ldots q_n\,,\, p_1\ldots p_n)$ in which the symplectic form $\omega$ is expressed as $J:=\tiny\left[\begin{array}{cc}0 & \id_n \\-\id_n & 0\end{array}\right]$, where $\id_n$ denotes the identity matrix of size $n$.  
In these coordinates, given a function $H\in C^{2}(U)\,:\,U\,\longrightarrow\,\mathbb{C}$ (the Hamiltonian), 
we define a \emph{Hamiltonian system} over $U\subset\mathbb{C}^{2n}$ as the differential equation given by the vector field 
	$$X_H:= J\nabla H = \sum_{i=1}^{n} \frac{\partial H }{\partial p_i} \frac{\partial}{\partial q_i} - \sum_{i=1}^{n} \frac{\partial H }{\partial q_i} \frac{\partial}{\partial p_i},$$
	corresponding to the Hamiltonian differential system
\begin{equation}\label{(1)}
\begin{array}{ccc}
\dot{q}_i = \frac{\partial H }{\partial p_i}(q\,,\,p),& \dot{p}_i = -\frac{\partial H }{\partial q_i}(q\,,\,p),& \text{for} \,\, i=1\ldots n.
\end{array}
\end{equation}

Consider a non-punctual integral curve $\Gamma$ of (\ref{(1)}). 
A meromorphic function $F\,:\, U\, \longrightarrow \,\mathbb{C}$ is called a \emph{meromorphic first integral} of  (\ref{(1)}) along 
$\Gamma$ if it is constant 
	along integral curves in a neighborhood of
	$\Gamma$, or equivalently when  $X_H(F) =0$. 
Observe that the Hamiltonian is a first integral of (\ref{(1)}), as we clearly have $X_H(H)=0$.

The Poisson bracket $\lbrace \,,\,\rbrace$ of two meromorphic functions $f, g\in C^{2}(U)$  is defined by $\lbrace f \,,\, g \rbrace:=\langle \nabla f \,,\, J\nabla g\rangle$. 
In  the  Darboux coordinates, its expression is ${\lbrace f \,,\, g \rbrace = \displaystyle\sum^{n}_{i=1} \frac{\partial f}{\partial q_i}\frac{\partial g}{\partial p_i}-\frac{\partial f}{\partial p_i}\frac{\partial g}{\partial q_i}}$. 
The Poisson bracket endows the set of first integrals with a structure of Lie algebra. 
A function $F$ is a first integral of (\ref{(1)}) if and only if $\lbrace F\,,\, H\rbrace=0$, i.e. $H$ and $F$ are \emph{in involution}.
Also, note that $X_{\lbrace F\,,\, H\rbrace} = [ X_F, X_H ]$, so the involution condition means that the associated Hamiltonian vector fields commute.

A Hamiltonian system with $n$ degrees of freedom is called \emph{Liouville integrable by meromorphic first integrals} along the integral curve $\Gamma$ if it possesses $n$ first integrals (including the Hamiltonian) meromorphic over $U$ which are functionally independent and in pairwise involution.

\subsection{Variational Equations}\label{subsection:variational equations}
Among the various  approaches to the study of meromorphic integrability of complex Hamiltonian systems,
 we choose a  Ziglin-Morales-Ramis type of approach. 
 Concretely, our starting points are the Morales-Ramis theorem  \cite{MoRa01b}  and its generalization, the Morales-Ramis-Sim\'o theorem  \cite{MRS,MoR}. These two results give necessary conditions for the meromorphic integrability of Hamiltonian systems. 
 Here, we need to introduce the notion of variational equation of order $p\in\mathbb{N}^{*}$ along a non-punctual integral curve of (\ref{(1)}).

Let $\Phi(z,t)$ be the flow defined by the equation (\ref{(1)}).
Given a non-punctual integral curve $\Gamma$ of (\ref{(1)}) and $z_0 \in \Gamma$, 
we let $\phi(t):=\Phi(z_0 \,,\, t)$  denote a temporal parametrization of  $\Gamma$.
We define  the \emph{$p^{th}$ variational equation}  $(\mathrm{VE}_{\phi}^{p})$ of (\ref{(1)})  along $\Gamma$ to be the differential equation satisfied by the $\xi_{j}:=\frac{\partial^{j}\Phi(z\,,\,t)}{\partial z^j}$ for $j\leq p$. 
For instance, the first three variational equations are given by (see \cite{MRS}, $\S 3.4$, Equation $(14)$, Page 860):
$$
(\mathrm{VE}_{\phi}^{3}): \; 
	\left\{ \begin{array}{cl}
		(\mathrm{VE}_{\phi}^{2}) : & 
			\left\{\begin{array}{l}
 				(\mathrm{VE}_{\phi}^{1}) :  \dot{\xi}_1 = d_{\phi} X_H \xi_1 \\
				\hspace{1.35cm}\dot{\xi}_2= d^{2}_{\phi}X_H(\xi_1\,,\, \xi_1) + d_{\phi}X_H \xi_2
			\end{array}
			\right.
		\\
		& \hspace{0.49cm}\hspace{1.35cm} \dot{\xi}_3  =  d^{3}_{\phi}X_H(\xi_1\,,\, \xi_1\,,\, \xi_1) + 2 d^2_{\phi}X_H (\xi_1\,,\,\xi_2) + d_{\phi} X_H \xi_3.
	\end{array}\right.
$$ 
For $p=1$, the first variational equation $(\mathrm{VE}_{\phi}^{1})$ is a linear differential equation
$$\dot{\xi}_1 = A_{1} \xi_1\text{   where   }A_{1}:= d_{\phi} X_H= J\cdot Hess_{\phi}(H)\in\mathfrak{sp}(n\,,\, \mathbb{C}\langle \phi(t) \rangle),$$
where $\mathbb{C}\langle \phi(t) \rangle$ denotes the differential field generated by the coefficients of  the parametrization $\phi(t)$.
Higher order variational equations are not linear  for $p\geq 2$. 
However, for every $(\mathrm{VE}_{\phi}^{p})$, one can construct  an equivalent linear differential system  $(\mathrm{LVE}_{\phi}^{p})$ called the {\em linearized} $p^{th}$ {\em variational equation} (see \cite{MRS}, $\S 3.4$ and \cite{Si14b}).
Indeed,  $(\mathrm{VE}_{\phi}^{p})$ is linear in $\xi_p$ and polynomial in the $\xi_i$ for $i<p$; 
however, the $\xi_i$ for $i<p$ are solutions of the linear differential system $(\mathrm{LVE}_{\phi}^{p-1})$
so that polynomials in the $\xi_i$ also satisfy linear differential systems, obtained via symmetric powers and tensor constructions.
See, for example, $\S 3$ of \cite{ApCoWe13a} for practical details on these tensor constructions on differential systems.\\
For example, $(\mathrm{VE}_{\phi}^{2})$ is linear in $\xi_2$ and linear in the monomials of degree $2$ in the
$\xi_1$, i.e. in the solutions of the second symmetric power system $Y'=\sym^2(A_{1}) Y$. 
Hence the system $(\mathrm{LVE}_{\phi}^{2})$ is lower block-triangular and its diagonal blocks are $\sym^2(A_{1})$ and $A_{1}$. 
We obtain   (see e.g.   \cite{MoR,ApWe11a,Si14b,CaWe15a}) 
the following matrices $A_p$ for the first $(\mathrm{LVE}_{\phi}^{p})$:
$$
A_2(x) =\left(\begin{array}{c|c} \sym^2\left(A_1(x)\right) & 0 \\\hline S_{2}(x) & A_1(x) \\ \end{array}\right), 
$$
$$
A_3(x) = \left(\begin{array}{c|c} \sym^3\left(A_1(x)\right) & 0 \\\hline S_{3}(x) & A_2(x) \\ \end{array}\right)
=
\left(\begin{array}{c|c|c}\sym^3\left(A_1(x)\right) & 0 & 0 \\\hline S_{3,2}(x) & \sym^2\left(A_1(x)\right) & 0 \\\hline S_{3,1}(x) & S_{2}(x) & A_1(x) \\ \end{array}\right).
$$
In general, the matrix of $(\mathrm{LVE}_{\phi}^{p})$ is of the form
$$
A_p(x) =\left(\begin{array}{c|c} \sym^p\left(A_1(x)\right) & 0 \\\hline S_{p}(x) & A_{p-1}(x) \\\end{array}\right).
$$
\\
In \cite{Si14b}, $\S 4.1$, Simon provides explicit formulas for these linearized variational equations. In what follows,
we will identify $(\mathrm{VE}_{\phi}^{p})$ and $(\mathrm{LVE}_{\phi}^{p})$ and we will just speak of variational equations of order $p$. 
\\
The matrix $\sym^i\left(A_1(x)\right) $ has $\binom{n+i-1}{n-1}$ rows and columns, so that 
$(\mathrm{LVE}_{\phi}^{p})$  is a first order linear differential system of $ {d_p:= \tiny\sum^{p}_{i=1} \binom{n+i-1}{n-1} = \binom{n+p}{n} -1}$ equations. 
The size $d_p$ grows fairly fast (polynomially of degree $n$ in $p$) and forbids the use of a generic algorithm to compute on  $(\mathrm{LVE}_{\phi}^{p})$.
For this reason, we  elaborate a specific algorithm which takes advantage of the structure of $(\mathrm{LVE}_{\phi}^{p})$
so that the polynomial growth of the size will become  a relatively minor concern.

\pagebreak[3]
\subsection{Differential Galois Theory and Reduced Forms}\label{sec22}

We begin this subsection with elements of differential Galois theory. We refer to \cite{PS03} or \cite{CrHa11a,Si09a} for details and proofs.

\subsubsection{The Base Field} \label{base-field}
Our base field will be $\mathbf{k}:= \CC\langle \phi(t) \rangle$, the differential field generated by the coefficients of  the parametrization $\phi(t)$ (and $\CC$ is the field of constants, which is assumed to be algebraically closed). 
We need to make assumptions about $\mathbf{k}$ to elaborate our algorithms. 
First we assume that $\mathbf{k}$ is an effective field, i.e. that one can compute representatives of the four operations $+,-,\times,/$ and one can effectively test whether two elements of $\mathbf{k}$ are equal, see e.g. \cite{Si91a}. Secondly, we assume that, given any scalar linear differential equation $L(y(x))=0$ where
	$$L(y(x)):= a_{n}(x) y^{(n)}(x) + a_{n-1}(x) y^{(n-1)}(x) + \cdots + a_{1}(x) y'(x)+ a_{0}(x) y(x), 
	\textrm{ with } a_i(x) \in \mathbf{k}, $$
one can effectively compute a basis of its space of \emph{rational solutions}, i.e. the solutions which are in  the base field $\mathbf{k}$.
The standard example of such a field would be $\mathbf{k}=\CC(x)$ with $\CC=\overline{\mathbb{Q}}$. 
Singer showed, in \cite{Si91a}, Lemma 3.5, that if $\mathbf{k}$ is an elementary extension
of $\CC(x)$ or if $\mathbf{k}$  is an algebraic extension of a purely transcendental Liouvillian extension
of $\CC(x)$, then $\mathbf{k}$ satisfies the above two conditions and hence suits our purposes. 
He also proved, see Theorem~4.1 in\cite{Si91a}, that an algebraic extension of $\mathbf{k}$ still satisfies our two assumptions, which will be useful, as reducing the first variational equation  may induce  algebraic extensions.
\subsubsection{Differential Galois Theory}
 Let us consider a linear differential system of the form $Y'(x)=A(x)Y(x)$, with $A(x)\in \CM_{\mathrm{n}}(\mathbf{k})$, that is a square matrix of size $n\in \mathbb{N}^{*}$ in coefficients in $\mathbf{k}$. 
 A \emph{Picard-Vessiot extension} for $Y'(x)=A(x)Y(x)$ is a differential field extension $K|\mathbf{k}$, generated over $\mathbf{k}$ by the entries of a fundamental solution matrix and such that the field of constants of $K$ is $\CC$. 
 The Picard-Vessiot extension $K$ exists and is unique up to differential field isomorphism. 
  \\ \par 
 The \emph{differential Galois group} $G$ of $Y'(x)=A(x)Y(x)$ is the group of field automorphisms of the Picard-Vessiot extension $K$ that commute with the derivation and leave all elements of $\mathbf{k}$ invariant. 
Let $U(x)\in \mathrm{GL}_{\mathrm{n}}(K)$ be a fundamental solution matrix of $Y'(x)=A(x)Y(x)$ with coefficients in $K$. 
 For any $ \varphi\in G$, $\varphi(U(x))$ is also a fundamental solution matrix,
 so there exists a constant matrix
 $C_{\varphi} \in  \mathrm{GL}_{\mathrm{n}}\left(\CC\right)$ such that $\varphi(U(x)) = U(x).C_{\varphi}$.
 The map $\rho_{U}:  \varphi \longmapsto C_{\varphi}$ 
is an injective group morphism. 
An important fact is that $G$, identified with $\hbox{Im } \rho_{U}$, may be viewed as a linear algebraic subgroup of~$\mathrm{GL}_{\mathrm{n}}\left(\CC\right)$. 
\\\par 
Two linear differential equations~$Y'(x)=A(x)Y(x)$ and~$Y'(x)=B(x)Y(x)$, with ${A(x),B(x)\in \CM_{\mathrm{n}}(\mathbf{k})}$  are said to be \emph{equivalent} over~$\mathbf{k}$ (or \emph{gauge equivalent} over~$\mathbf{k}$) 
when there exists~$P(x)\in \mathrm{GL}_{\mathrm{n}}(\mathbf{k})$, called a \emph{gauge transformation matrix}, such that $$B(x)=P(x)\left[ A(x)\right]:=P(x)A(x)P^{-1}(x)+P'(x)P^{-1}(x).$$
 Note that in this case: $$Y'(x)=A(x)Y(x)\Longleftrightarrow \left[P(x)Y(x)\right]'=B(x)P(x)Y(x).$$ 
Conversely, if there exist matrices $A(x),B(x)\in \CM_{\mathrm{n}}(\mathbf{k})$  and~$P(x)\in \mathrm{GL}_{\mathrm{n}}(\mathbf{k})$, such that we have
$Y'(x)=A(x)Y(x)$, ${Z'(x)=B(x)Z(x)}$ and $Z(x)=P(x)Y(x)$, then 
$$B(x)=P(x)\left[ A(x)\right].$$
The \emph{Lie algebra} $\mathfrak{g}$ of the linear algebraic group $G\subset \mathrm{GL}_{\mathrm{n}}\left(\CC\right)$
is the tangent space to $G$ at the identity. Equivalently, it is the set of matrices $N$ such that $\id_n + \epsilon N$ 
satisfies the defining equations of the algebraic group $G$ modulo $\epsilon^2$.

Part two of the following proposition is known as the Kolchin-Kovacic reduction theorem.
A proof can be found in \cite{PS03}, Proposition 1.31 and Corollary 1.32. See also \cite{BlMo10a}, Theorem~5.8.
\begin{proposition}[Kolchin-Kovacic reduction theorem]\label{propo1}
Let us consider the differential system $Y'(x)=A(x)Y(x)$ with $A(x)\in \CM_{\mathrm{n}}(\mathbf{k})$.
Let $G$ be its differential Galois group and $\mathfrak{g}$ be the Lie algebra of $G$.
\begin{enumerate}
\item Let $H\subset \mathrm{GL}_{\mathrm{n}}\left(\CC\right)$ be a linear algebraic group and $\mathfrak{h}\subset \CM_{\mathrm{n}}\left(\CC\right)$ be its Lie algebra. If $A(x)$ belongs to $\mathfrak{h}(\mathbf{k}):=\mathfrak{h}\otimes_{\CC}\mathbf{k}$, then $G$ is contained in a conjugate of $H$.
\item Assume that $\mathbf{k}$ is a $\mathcal{C}^{1}$-field \footnote{A field $\mathbf{k}$ is a $\mathcal{C}^{1}$-field when every non-constant homogeneous polynomial $P$ over $\mathbf{k}$ has a non-trivial zero provided that the number of its variables is more than its degree. For example, $\CC(x)$ is a $\mathcal{C}^{1}$-field and any algebraic extension of a $\mathcal{C}^{1}$-field is a $\mathcal{C}^{1}$-field.} and $G$ is connected. Let $H \supset G$ be a connected linear algebraic group with Lie algebra $\mathfrak{h}$ such that $A(x)\in \mathfrak{h}(\mathbf{k})$. 
Then, there exists a gauge transformation $P(x)\in H(\mathbf{k})$ such that $P(x)[A(x)]\in \mathfrak{g}(\mathbf{k})$.
\end{enumerate}
\end{proposition}

\subsubsection{Reduced Forms of Linear Differential Systems} \label{section-reduced}
Let  $A(x)\in \CM_{\mathrm{n}}\left(\mathbf{k}\right)$, $G$ be the differential Galois group of $Y'(x)=A(x)Y(x)$
and $\mathfrak{g}$ its Lie algebra.\\
We say that the system $Y'(x)=A(x)Y(x)$ is \emph{in reduced form} (or in \emph{Kolchin-Kovacic reduced form})
when $A(x) \in \mathfrak{g}(\mathbf{k})=\mathfrak{g}\otimes_{\CC} \mathbf{k}$. This section contains a quick survey on reduced forms and their practical use.

Following \cite{WeNo63a}, a  \emph{Wei-Norman decomposition} of $A(x)$ is a finite sum of the form
$$A(x)=\sum a_{i}(x)M_{i},$$
where $M_{i}$ has coefficients in $\CC$ and the $a_{i}(x)\in \mathbf{k}$ form a basis of the $\CC$-vector space spanned by the entries of $A(x)$. 
The $M_{i}$ depend on the choice of $a_{i}(x)$ but the $\CC$-vector space generated by the $M_{i}$ is independent of the choice of the $a_{i}(x)$.
\begin{definition}  
Let $\mathrm{Lie}(A)\subset \CM_{\mathrm{n}}\left(\CC\right)$ 
denote the Lie algebra generated by the $M_{i}$. 
We define $\mathrm{Lie_{alg}}(A)\subset \CM_{\mathrm{n}}\left(\CC\right)$, called \emph{the Lie algebra associated to $A$},  as the algebraic envelope of the Lie algebra $\mathrm{Lie}(A)$, i.e.
as the smallest Lie algebra of a linear algebraic group which contains  $\mathrm{Lie}(A)$. 
\end{definition}
Let $\mathrm{Lie}(A;\mathbf{k}):=\mathrm{Lie}(A)(\mathbf{k})\subset \CM_{\mathrm{n}}\left(\mathbf{k}\right)$ and 
	${\mathrm{Lie_{alg}}(A;\mathbf{k}):=\mathrm{Lie_{alg}}(A)(\mathbf{k})\subset \CM_{\mathrm{n}}\left(\mathbf{k}\right)}$.
We see that the system $Y'(x)=A(x)Y(x)$ is in reduced form when $\mathrm{Lie_{alg}}(A;\mathbf{k}) = \mathfrak{g}(\mathbf{k})$.\\\par 
 
These reduced forms have long been studied in the context of inverse problems in differential Galois theory (see \cite{MiSi02a} and references therein). Their use in direct problems is more recent. 
Blazquez and Morales use them in their studies of Lie-Vessiot systems in \cite{BlMo10a,BlMo12a}.
Their application to Morales-Ramis theory is initiated in \cite{ApWe12b} where Aparicio-Monforte and Weil show how to put the first variational equation in reduced form. In \cite{ApCoWe13a}, the same authors with Compoint show that a system is in reduced form if and only if, for any tensor construction $\mathtt{const}(A(x))$ on $A(x)$, any rational or hyperexponential solution of $Y'= \mathtt{const}(A(x)) Y$ has constant coefficients. One can also find in \cite{ApCoWe13a} a complete procedure to put a  linear differential system into reduced form when it is irreducible (or completely reducible). 
This does not really apply here as the variational equations are generally reducible (and not completely reducible) systems -- and the reduction method of \cite{ApCoWe13a} is far from being efficient yet.\\ \par 

The approach that we elaborate in this paper was initiated (incompletely) in \cite{ApWe11a}. It is based on another criterion for reduced form, which
is given in the following lemma.
  
  \begin{lemma} \label{reduite-minimale}
Given $A(x)\in \CM_{\mathrm{n}}\left(\mathbf{k}\right)$, let $G$ be the differential Galois group of $Y'(x)=A(x)Y(x)$
and $\mathfrak{g}$ be its Lie algebra. Let $H$ be a connected linear algebraic group whose Lie algebra $\mathfrak{h}$
satisfies $\mathfrak{h}=\mathrm{Lie_{alg}}(A)$.
Assume that $G$ is connected. 
\\
Then $Y'(x)=A(x)Y(x)$ is in reduced form, i.e. $G=H$ and $\mathfrak{g}=\mathfrak{h}$,
if and only if, for all gauge transformation matrices $P(x)$ in $H(\mathbf{k})$, we have $\mathfrak{h}(\mathbf{k})=\mathrm{Lie_{alg}}(P[A];\mathbf{k})$.
  \end{lemma}
  \begin{proof}
  Follows directly from the Kolchin-Kovacic reduction theorem, see Proposition \ref{propo1}.
  \end{proof}

\subsection{The Morales-Ramis-Sim\'o Integrability Criterion.}
We are now in position to state the Morales-Ramis-Sim\'o integrability criterion. See \cite{MRS} for a proof and $\S \ref{sec2}$ for the definitions. 
\pagebreak[3]
\begin{theorem}[Morales-Ramis-Sim\'o integrability criterion]
Consider a Hamiltonian vector field $X_H$ and a non-punctual integral curve $\Gamma$.
For $p\in \mathbb{N}^{*}$, let $G_p$ be the differential Galois group of $(\mathrm{VE}_{\phi}^{p})$, the $p^{th}$ variational equation along $\Gamma$. Let $\mathfrak{g}_{p}$ be the Lie algebra of $G_p$.
Assume that the Hamiltonian vector field $X_H$ is Liouville integrable by meromorphic first integrals along the integral curve $\Gamma$. Then, for all $p\in \mathbb{N}^{*}$, $\mathfrak{g}_{p}$ is abelian.
\end{theorem}

Of course, given $p\in \mathbb{N}^{*}$, computing the differential Galois group $G_p$ of such a big differential system would be an unrealistic task in practice unless we use the structure of the system to simplify the computations.
We will establish a specific \emph{reduction method}, i.e. compute a gauge transformation
matrix $P_p(x)$ such that $P_p (x)[A_p(x)] \in \mathfrak{g}_p(\mathbf{k})$.
After this reduction process, the Lie algebra $\mathfrak{g}_p$ is
easily read and its abelianity (or not) is given in the process. 
Furthermore, if $\mathfrak{g}_p$ is abelian, then this process will have prepared the system to allow an efficient reduction of the next variational equation.

\subsection{The Strategy for an effective Morales-Ramis-Sim\'o Criterion.}\label{sec25}

We refer to $\S \ref{subsection:variational equations}$ and $\S \ref{sec22}$ for the notations used in this subsection. Let us fix an integer $p\geq 2$. The matrix of the $p^{th}$ variational equation has the form 
$$
A_p(x) =\left(\begin{array}{c|c} \sym^p\left(A_1(x)\right) & 0 \\\hline S_{p}(x) & A_{p-1}(x) \\ \end{array}\right).
$$
For each $m\in \{1,\ldots, p\}$, we let $G_m$ denote the differential Galois group of the $m^{th}$ variational equation $Y'(x)=A_m(x) Y(x)$
and $\mathfrak{g}_m$ its Lie algebra. For all $m \in \{1,\ldots, p-1\}$, we assume that 
we know a gauge transformation matrix $P_m(x)$ such that 
$P_m(x)[A_m(x)]$ is in reduced form, i.e. $\mathrm{Lie_{alg}}(P_m[A_m])= \mathfrak{g}_m $, 
and we further assume that each $\mathfrak{g}_m$ is abelian.
We let $A_{m,red}(x)$ denote the obtained reduced form, that is $A_{m,red}(x) := P_m(x)[A_m(x)]$.

Under these hypotheses, we will show in the next section how to put the $p^{th}$ variational equation $A_p(x)$ into reduced form in an efficient way.

\begin{remark}
Our assumptions imply that the first variational equation is in reduced form. 
This in turn implies that our base field $\mathbf{k}$ is no longer just $\CC\langle \phi \rangle$ but may be an algebraic extension of the latter (see \cite{ApCoWe13a}). In the sequel, our base field
$\mathbf{k}$  is the algebraic extension of  $\CC\langle \phi \rangle$ which is needed to put the first variational equation into reduced form. 
Since an algebraic extension of a $\mathcal{C}^{1}$-field is a $\mathcal{C}^{1}$-field, we obtain that $\mathbf{k}$ is a $\mathcal{C}^{1}$-field provided that  $\CC\langle \phi \rangle$ is a $\mathcal{C}^{1}$-field. Consequently, we are allowed to use Proposition \ref{propo1} as soon as  $\CC\langle \phi \rangle$ is a $\mathcal{C}^{1}$-field. From now on, we assume that $\mathbf{k}$ is a $\mathcal{C}^{1}$-field.
\end{remark}

Our assumptions also imply (see \cite{ApCoWe13a}, Lemma 32, Page 1513) that, for all ${m\in \{1,\ldots, p-1\}}$, the differential Galois groups 
$G_m$ are connected. Moreover, both the groups $G_m$ and their Lie algebras $\mathfrak{g}_m$ are abelian.

\begin{lemma} \label{gp-connected}
The group $G_{p}$ is connected.
\end{lemma}
\begin{proof}
This is a direct application of \cite{MoR}, Lemma 10.
\end{proof}

As we can see in \cite{ApCoWe13a}, Lemma 14, Page 1508, $$\Sym^{p}(P_1(x))[\sym^p(A_1(x))] = \sym^p(A_{1,red}(x)).$$
Also, $\sym^p(A_{1,red}(x))$ is a reduced form of $\sym^p(A_{1}(x))$. Indeed, this follows from \cite{ApCoWe13a}, Theorem 1,
because any tensor construction on $\sym^p(A_{1}(x))$ is a construction on $A_1(x)$.
\\
Consider the block-diagonal gauge transformation matrix 
$$Q(x) := \left(\begin{array}{c|c} \Sym^{p}(P_1(x)) & 0 \\\hline 0 & P_{p-1}(x) \\ \end{array}\right).$$
Thanks to the above remarks (see also \cite{ApT}, $\S 4.5.2$), we find that 
$$ Q(x)[A_p(x)] = 
	\left(\begin{array}{c|c} \sym^p\left(A_{1,red}(x)\right) & 0 \\\hline S(x) & A_{p-1,red}(x) \\\end{array}\right),$$
where $S (x)$ has entries in $\mathbf{k}$, and the block-diagonal part of $ Q(x)[A_p(x)]$ is in reduced form. Furthermore,
$ \mathrm{Lie_{alg}}\left(   
\begin{array}{c|c} \sym^p\left(A_{1,red}\right) & 0 \\\hline 0 & A_{p-1,red} \\ \end{array}
 \right) 
 \textrm{ is abelian.}
 $

\section{Reduction of Linear Differential Systems with a Reduced Abelian Diagonal Part}\label{sec3}
 
The previous subsection shows that finding a reduced form for the $p^{th}$ variational equation now amounts to finding a reduced form
for 
$$ A(x) := Q(x)[A_p (x)]  = \left(\begin{array}{c|c} \sym^p\left(A_{1,red}(x)\right) & 0 \\\hline S(x) & A_{p-1,red}(x) \\\end{array}\right)\in \CM_{\mathrm{n}}(\mathbf{k}).$$

The submatrices $\mathrm{\sym}^{p}\left(A_{1,red}(x)\right)$ and $A_{p-1,red} (x)$ belong respectively to $\CM_{\mathrm{n_{1}}}(\mathbf{k})$ and $\CM_{\mathrm{n_{2}}}(\mathbf{k})$, with $n_{1}:=\binom{n+p-1}{n-1}$ and $n_{2}:=\binom{n+p-1}{n}-1$. 
 \\
We have $A(x)= \adiag (x)+ \asub (x)$,  where  $\adiag (x):=\left(\begin{array}{c|c} \mathrm{\sym}^{p}\left(A_{1,red}(x)\right) & 0 \\\hline 0 & A_{p-1,red} (x)\end{array}\right)$ 
 and $\asub (x):=\left(\begin{array}{c|c} 0 & 0 \\\hline S(x)& 0\end{array}\right).$
 We have seen that $Y'(x)=\adiag(x)Y(x)$ is in reduced form and  $\mathrm{Lie_{alg}}(\adiag)$ is abelian.
The aim of this section is to show how to use those hypotheses to put the full system $Y' (x)=A(x)Y(x)$ in reduced form.

\subsection{The Diagonal and Off-Diagonal Subalgebras}\label{sec31}

We refer to $\S \ref{section-reduced}$ for the notations used in this subsection. 
Let $M_1,\ldots, M_\delta \in\CM_{\mathrm{n}}\left(\CC\right)$ be a basis of $\mathrm{Lie_{alg}}\left(\adiag\right)$ and let $B_1,\ldots, B_\sigma \in \CM_{\mathrm{n}}\left(\CC\right)$ be a basis of $\mathrm{Lie_{alg}}\left(\asub\right)$.
We define the vector space ${\hlie := \mathrm{Lie_{alg}}\left(\adiag\right) \oplus \mathrm{Lie_{alg}}\left(\asub\right)}$.
Note that $\mathrm{Lie_{alg}}(A)\subseteq \hlie$, and $\hlie$ is the Lie algebra of a linear algebraic group. Let us sum up some elementary properties of $\hlie$ in the two following lemmas: 

\begin{lemma}\label{diagsub}
Let us consider a matrix $\left(\begin{array}{c|c} N_{1}(x) & 0 \\ \hline N_{2,1}(x) & N_{2}(x)\end{array}\right)\in \mathfrak{h}(\mathbf{k})$ 
and matrices $\left(\begin{array}{c|c} 0 & 0 \\\hline C_{1}(x) & 0\end{array}\right),\left(\begin{array}{c|c} 0 & 0 \\\hline C_{2}(x) & 0\end{array}\right)\in\mathrm{Lie_{alg}}\left(\asub; \mathbf{k}\right)$. 
\begin{enumerate}
\item For $(i,j)\in \{1;2\}^{2}$, $\left(\begin{array}{c|c} 0 & 0 \\\hline C_{i}(x) & 0\end{array}\right)\left(\begin{array}{c|c} 0 & 0 \\\hline C_{j}(x) & 0\end{array}\right)=0$.
\item The matrix $\left(\begin{array}{c|c} N_{1}(x) & 0 \\\hline N_{2,1}(x) & N_{2}(x)\end{array}\right)\left(\begin{array}{c|c} 0 & 0 \\\hline C_{1}(x) & 0\end{array}\right) $ and the Lie bracket $\left[\left(\begin{array}{c|c} N_{1}(x) & 0 \\\hline N_{2,1}(x) & N_{2}(x)\end{array}\right),\left(\begin{array}{c|c} 0 & 0 \\\hline C_{1}(x) & 0\end{array}\right) \right]$ belong to $ \mathrm{Lie_{alg}}\left(\asub  ;\mathbf{k}\right)$. Furthermore $\mathrm{Lie_{alg}}\left(\asub ;\mathbf{k}\right)$ is an ideal in $\hlie (\mathbf{k})$.
\end{enumerate}
\end{lemma}

\begin{proof} 
\begin{enumerate}
 \item A straightforward computation shows the first point of the lemma.
 \item We have $\left(\begin{array}{c|c} N_{1}(x) & 0 \\\hline N_{2,1}(x) & N_{2}(x)\end{array}\right)\left(\begin{array}{c|c} 0 & 0 \\\hline C_{1}(x) & 0\end{array}\right)=\left(\begin{array}{c|c} 0 & 0 \\\hline N_{2}(x)C_{1}(x)&0  \end{array}\right)\in \mathfrak{h}(\mathbf{k})$ and $$\left[\left(\begin{array}{c|c} N_{1}(x) & 0 \\\hline N_{2,1}(x) & N_{2}(x)\end{array}\right),\left(\begin{array}{c|c} 0 & 0 \\\hline C_{1}(x) & 0\end{array}\right) \right]=\left(\begin{array}{c|c} 0 & 0 \\\hline N_{2}(x)C_{1}(x)-C_{1}(x)N_{1}(x) & 0\end{array}\right)\in \mathfrak{h}(\mathbf{k}) .$$
We prove that they belong to $\mathrm{Lie_{alg}}\left(\asub  ;\mathbf{k}\right)$ using that fact that the diagonal blocks of the two matrices
are zero.
 The latter Lie bracket identity also shows  that $\mathrm{Lie_{alg}}\left(\asub ;\mathbf{k}\right)$ is an ideal in $\hlie (\mathbf{k})$.
 \end{enumerate}
\end{proof}

\pagebreak[3]
\begin{lemma}\label{lem1}
For all $ B(x)\in \mathrm{Lie_{alg}}\left(\asub ;\mathbf{k}\right)$, we have $\exp(B(x)) = \id_{\mathrm{n}} + B(x)$ and ${\log(\id_{\mathrm{n}} + B(x))=B(x)}$. 
This induces two bijective maps which are inverses of each other
$$\begin{array}{cccc}
\exp : &  \mathrm{Lie_{alg}}\left(\asub ;\mathbf{k} \right)& \longrightarrow & \Big\{\id_{\mathrm{n}}+B(x), B(x)\in  \mathrm{Lie_{alg}}\left(\asub ;\mathbf{k}\right)\Big\} \\
& B(x)&\mapsto  &\id_{\mathrm{n}}+B(x)\\
 \log : &  \Big\{\id_{\mathrm{n}}+B(x), B(x)\in  \mathrm{Lie_{alg}}\left(\asub ;\mathbf{k}\right)\Big\} & \longrightarrow &  \mathrm{Lie_{alg}}\left(\asub ;\mathbf{k}\right)\\
 &\id_{\mathrm{n}}+B(x)&\mapsto  &B(x).
\end{array}$$
\end{lemma}

\begin{proof}
Let  $B(x)\in \mathrm{Lie_{alg}}\left(\asub;\mathbf{k}\right)$. 
The equality $\exp(B(x)) = \id_{\mathrm{n}} + B(x)$ is a direct consequence of the first point of Lemma \ref{diagsub}. The same argument shows that ${\log(\id_{\mathrm{n}} + B(x))=B(x)}$. It follows directly that $\exp$ and $\log$ are bijective on the wished sets and inverses of each other. 
\end{proof}

\pagebreak[3]
\subsection{The Shape of the Reduction Matrix}\label{sec32}

We refer to $\S \ref{sec22}$ and $\S \ref{sec31}$ for the notations and definitions used in this subsection. The aim of this subsection is to prove: 

\begin{theorem}\label{theo1}
There exists a gauge transformation $$P(x)\in\Big\{\id_{\mathrm{n}}+B(x), B(x)\in  \mathrm{Lie_{alg}}\left(\asub ;\mathbf{k}\right)\Big\},$$ such that ${Y'(x) = P(x)[A(x)]Y(x)}$ is in reduced form.
\end{theorem}
Let $G$ be the differential Galois group of $Y'(x)=A(x)Y(x)$. 
By construction, we have $G=G_p$, where $G_p$ is the differential Galois group of the p-th variational equation $Y'(x)=A_{p}(x)Y(x)$. 
Let $H$ be the connected linear algebraic group with Lie algebra $\hlie$. Before proving Theorem~\ref{theo1}, we start with a key lemma.

\pagebreak[3]
\begin{lemma}\label{lem2}
There exists a unipotent gauge transformation $P(x)$, of the form $P(x) = \left(\begin{array}{c|c} \id_{\mathrm{n_{1}}} & 0 \\\hline N(x) & \id_{\mathrm{n_{2}}}\end{array}\right)\in H (\mathbf{k})$,  such that $Y'(x) = P(x)[A(x)]Y(x)$ is in reduced form. 
\end{lemma}
 
 \begin{proof} 
Let $H_{A}$ be the connected linear algebraic group with Lie algebra $\mathrm{Lie_{alg}}\left(A;\mathbf{k}\right)$. 
We have the inclusions $G \subseteq H_{A} \subseteq H$.
As $G=G_p$, Lemma \ref{gp-connected} shows that $G$ is connected.
So we may use the second point of Proposition~\ref{propo1} to obtain the existence of ${\widetilde{Q}(x):=\left(\begin{array}{c|c} D_{1}(x) & 0 \\\hline S_{Q}(x) & D_{2}(x)\end{array}\right)\in H_{A}(\mathbf{k})}$ such that the linear differential system $Y'(x) = \widetilde{Q}(x)[A(x)]Y(x)$ is in reduced form. 
Let ${R(x):=\left(\begin{array}{c|c} D_{1}^{-1}(x) & 0 \\\hline 0 & D_{2}^{-1}(x)\end{array}\right)\in H (\mathbf{k})}$ so that ${R(x)\widetilde{Q}(x)=\left(\begin{array}{c|c} \id_{\mathrm{n_{1}}} & 0 \\\hline D_{2}^{-1}(x)S_{Q}(x) & \id_{\mathrm{n_{2}}}\end{array}\right)\in H (\mathbf{k})}$. 
Consequently, to prove the lemma, it is sufficient to prove that $Y'(x) = R(x)\widetilde{Q}(x)[A(x)]Y(x)$ is in reduced form. We have to prove that $\mathrm{Lie_{alg}}\left(\widetilde{Q}[A];\mathbf{k}\right)= \mathrm{Lie_{alg}}\left(R\widetilde{Q}[A];\mathbf{k}\right)$. 
Let $H_{R\widetilde{Q}}$ be the algebraic group whose Lie algebra is $\mathrm{Lie_{alg}}\left(R\widetilde{Q}[A]\right)$. 
Thanks to the first point of Proposition \ref{propo1}, the group $H_{R\widetilde{Q}}$ contains $G $. 
Since $Y'(x) = \widetilde{Q}(x)[A(x)]Y(x)$ is in reduced form, $G $ is an algebraic group whose Lie algebra is $\mathrm{Lie_{alg}}\left(\widetilde{Q}[A]\right)$. This implies that ${\mathrm{Lie_{alg}}\left(\widetilde{Q}[A];\mathbf{k}\right)\subseteq \mathrm{Lie_{alg}}\left(R\widetilde{Q}[A];\mathbf{k}\right)}$.\\\par 
 Let $K|\mathbf{k}$ denote the Picard-Vessiot extension for the equation $Y'(x)=A(x)Y(x)$ and let ${U(x):=\left(\begin{array}{c|c} U_{1}(x) & 0 \\\hline U_{2,1}(x) & U_{2}(x)\end{array}\right)\in \mathrm{GL}_{\mathrm{n}}(K)}$, with $U_{i}(x)\in \mathrm{GL}_{\mathrm{n_{i}}}\left(K\right)$ be a fundamental solution. 
 The elements of $G $ are of the form $\left(\begin{array}{c|c} G_{1} & 0 \\\hline G_{2,1} & G_{2}\end{array}\right)\in \mathrm{GL}_{\mathrm{n}}\left(\CC\right)$, with $G_{i}\in \mathrm{GL}_{\mathrm{n_{i}}}\left(\CC\right)$. 
 Let $G_{\mathrm{sub}}$ be the subgroup of elements of $G$ of the form $\left(\begin{array}{c|c} \id_{\mathrm{n_{1}}} & 0 \\\hline G_{2,1} & \id_{\mathrm{n_{2}}}\end{array}\right)$. A direct computation shows that $G_{\mathrm{sub}}$ is a normal subgroup of $G$. Therefore, $G\simeq G_{\mathrm{sub}}\rtimes G/G_{\mathrm{sub}}$. Due to \cite{PS03}, Proposition 1.34, (2), ${G_{\mathrm{diag}}:=G/ G_{\mathrm{sub}}}$ is isomorphic to the differential Galois group of $Y'(x)=\adiag (x)Y(x)$. 
 Let us write $\widetilde{Q}(x)[A(x)]=:\left(\begin{array}{c|c} D_{1}(x)[\mathrm{\sym}^{p}\left(A_{1,red}(x)\right)] & 0 \\\hline \underline{A}_{2,1}(x) & D_{2}(x)[A_{p-1,red}(x)]\end{array}\right),$ for some matrix $\underline{A}_{2,1}(x)$ in coefficients in $\mathbf{k}$. 
 We use the relation $G\simeq G_{\mathrm{sub}}\rtimes G_{\mathrm{diag}}$ and the fact that $Y'(x)=\widetilde{Q}(x)[A(x)]Y(x)$ is in reduced form to find that
 $$
 \begin{array}{lll}
 \mathrm{Lie_{alg}}\left(\widetilde{Q}[A];\mathbf{k}\right)&\simeq &\mathrm{Lie_{alg}}\left(\begin{array}{c|c} D_{1}[\mathrm{\sym}^{p}\left(A_{1,red}\right)] & 0 \\\hline 0 & D_{2}[A_{p-1,red}]\end{array}\right)(\mathbf{k})\oplus \mathrm{Lie_{alg}}\left(\begin{array}{c|c} 0 & 0 \\\hline \underline{A}_{2,1} & 0\end{array}\right)(\mathbf{k}).
 \end{array}
 $$
A direct computation shows that 
 \begin{equation}\label{eq5}
 R(x)\widetilde{Q}(x)[A(x)]=\left(\begin{array}{c|c} \mathrm{\sym}^{p}\left(A_{1,red}(x)\right) & 0 \\\hline D_{2}^{-1}(x)\underline{A}_{2,1}(x)D_{1}(x) & A_{p-1,red}(x)\end{array}\right).
 \end{equation}
 By construction, 

$$
\begin{array}{lll}
\mathrm{Lie_{alg}}\left(R\widetilde{Q}[A];\mathbf{k}\right)&\subseteq&\mathrm{Lie_{alg}}\left(\begin{array}{c|c} \mathrm{\sym}^{p}\left(A_{1,red}\right) & 0 \\\hline 0 & A_{p-1,red}\end{array}\right)(\mathbf{k})\oplus \mathrm{Lie_{alg}}\left(\begin{array}{c|c} 0 & 0 \\\hline D_{2}^{-1}\underline{A}_{2,1}D_{1} & 0\end{array}\right)(\mathbf{k}).

\end{array}
$$

Since $D_{1}(x)$ and $D_{2}(x)$ are invertible matrices, $\mathrm{Lie_{alg}}\left(\begin{array}{c|c} 0 & 0 \\\hline \underline{A}_{2,1} & 0\end{array}\right)(\mathbf{k})$ and $\mathrm{Lie_{alg}}\left(\begin{array}{c|c} 0 & 0 \\\hline D_{2}^{-1}\underline{A}_{2,1}D_{1} & 0\end{array}\right)(\mathbf{k})$ have the same dimension. Due to the inclusion $\mathrm{Lie_{alg}}\left(\widetilde{Q}[A];\mathbf{k}\right)\subseteq \mathrm{Lie_{alg}}\left(R\widetilde{Q}[A];\mathbf{k}\right)$ we obtain that 
\begin{equation}\label{eq6}
\mathrm{Lie_{alg}}\left(\begin{array}{c|c} 0 & 0 \\\hline \underline{A}_{2,1} & 0\end{array}\right)(\mathbf{k})=\mathrm{Lie_{alg}}\left(\begin{array}{c|c} 0 & 0 \\\hline D_{2}^{-1}\underline{A}_{2,1}D_{1} & 0\end{array}\right)(\mathbf{k}).
\end{equation}
 
Using the facts that the systems $Y'(x)=\adiag (x)Y(x)$ and $Y' (x)= \widetilde{Q}(x)[A(x)]Y(x)$ are in reduced form and $G\simeq G_{\mathrm{sub}}\rtimes G_{\mathrm{diag}}$, we find that $${\mathrm{Lie_{alg}}\left(\begin{array}{c|c} \mathrm{\sym}^{p}\left(A_{1,red}\right) & 0 \\\hline 0 & A_{p-1,red}\end{array}\right)(\mathbf{k})=\mathrm{Lie_{alg}}\left(\begin{array}{c|c} D_{1}[\mathrm{\sym}^{p}\left(A_{1,red}\right)] & 0 \\\hline 0 & D_{2}[A_{p-1,red}]\end{array}\right)(\mathbf{k})}.$$ Combined with (\ref{eq6}), this proves that $\mathrm{Lie_{alg}}\left(R\widetilde{Q}[A];\mathbf{k}\right)\subseteq\mathrm{Lie_{alg}}\left(\widetilde{Q}[A];\mathbf{k}\right)$. Since we have an inclusion ${\mathrm{Lie_{alg}}\left(\widetilde{Q}[A];\mathbf{k}\right)\subseteq\mathrm{Lie_{alg}}\left(R\widetilde{Q}[A];\mathbf{k}\right)}$, we obtain the equality $\mathrm{Lie_{alg}}\left(R\widetilde{Q}[A];\mathbf{k}\right)=\mathrm{Lie_{alg}}\left(\widetilde{Q}[A];\mathbf{k}\right)$. In other words, $Y'(x) = R(x)\widetilde{Q}(x)[A(x)]Y(x)$ is in reduced form. 
 \end{proof}
\begin{proof}[Proof of Theorem \ref{theo1}.]
It follows from Lemma \ref{lem2} that a reduction matrix can always be chosen of the form $P(x)=\left(\begin{array}{c|c} \id_{\mathrm{n_{1}}} & 0 \\\hline N(x) & \id_{\mathrm{n_{2}}}\end{array}\right)\in H (\mathbf{k})$,  where $N(x)\in \CM_{\mathrm{n_{2},n_{1}}}\left(\mathbf{k}\right)$. By a straightforward computation, we find $\log(P(x))=\left(\begin{array}{c|c} 0 & 0 \\\hline N(x) &0\end{array}\right)\in \hlie (\mathbf{k})$. But with the same reasoning as in the proof of Lemma \ref{diagsub}, we obtain that $\log(P(x))\in \mathrm{Lie_{alg}}\left(\asub;\mathbf{k}\right)$. This concludes the proof of Theorem~\ref{theo1}.
\end{proof}
The following corollary will be crucial for the reduction procedure of $\S \ref{sec34}$. 

\begin{corollary}\label{coro1}
Assume that, for all gauge transformations of the form ${P(x)\in\Big\{\id_{\mathrm{n}}+B(x), B(x)\in  \mathrm{Lie_{alg}}\left(\asub ;\mathbf{k}\right)\Big\}}$, we have $\mathrm{Lie}(A;\mathbf{k})=\mathrm{Lie}(P[A];\mathbf{k})$. Then, $Y'(x)=A(x)Y(x)$ is in reduced form.
\end{corollary}

\begin{proof}
	Theorem~\ref{theo1} provides a ${B(x)\in \mathrm{Lie_{alg}}\left(\asub ;\mathbf{k}\right)}$ 
	and $P(x)=\id_{\mathrm{n}}+B(x)$ such that the system $Y'(x)=P(x)[A(x)]Y(x)$  is in reduced form. 
In virtue of the hypothesis, ${\mathrm{Lie}(A;\mathbf{k})=\mathrm{Lie}(P[A];\mathbf{k})}$. 
This implies that ${\mathrm{Lie_{alg}}(A;\mathbf{k}) =\mathfrak{g}(\mathbf{k})}$, where $\mathfrak{g}$ is the Lie algebra of the differential Galois group $G$ of $Y'(x)=A(x)Y(x)$. 
This proves that $Y'(x)=A(x)Y(x)$ is in reduced form.
\end{proof}
\pagebreak[3]
\subsection{The Adjoint Action}\label{sec33}
We refer to $\S \ref{sec22}$ and $\S \ref{sec31}$ for the notations and definitions used in this subsection. 
In $\S \ref{sec32}$, we have proved  the existence of a gauge transformation matrix $P(x)\in\Big\{\id_{\mathrm{n}}+B(x), B(x)\in  \mathrm{Lie_{alg}}\left(\asub ;\mathbf{k}\right)\Big\}$, such that ${Y'(x) = P(x)[A(x)]Y(x)}$ is in reduced form. 
Let $B_1,\ldots, B_\sigma \in \CM_{\mathrm{n}}\left(\CC\right)$ denote a basis of $\mathrm{Lie_{alg}}\left(\asub\right)$. 

\begin{proposition}\label{propo2}
If $P(x):=\id_{\mathrm{n}} + \displaystyle\sum_{i=1}^{\sigma} f_i (x)B_i$, with $f_i (x)\in \mathbf{k}$
	and $B_i\in \mathrm{Lie_{alg}}\left(\asub\right)$, then
$$P(x)[A(x)] = A(x) + \sum_{i=1}^{\sigma} f_i (x)[B_i,\adiag (x)] - \sum_{i=1}^{\sigma} f_i' (x) B_i.$$
\end{proposition}

\begin{proof}
Due to the first point of Lemma \ref{diagsub},  we have the equalities ${P^{-1} (x)= \id_{\mathrm{n}} - \displaystyle\sum_{i=1}^{\sigma} f_i (x)B_i }$ and ${P(x)A(x) =  A(x) + \displaystyle\sum_{i=1}^{\sigma} f_{i} (x)B_i \adiag (x)}$. 
As $A(x)=\adiag (x)+\asub (x)$, we use Lemma \ref{diagsub} and find that

$$\begin{array}{lll}
P(x)A(x) P^{-1}(x) &=& \left(\adiag (x)+\asub (x) +\displaystyle\sum_{j=1}^{\sigma} f_{j}(x) B_j \adiag (x)\right)\left(\id_{\mathrm{n}} - \displaystyle\sum_{k=1}^{\sigma} f_k (x)B_k\right)\\
&= &A(x) + \displaystyle\sum^{\sigma}_{j=1} f_{j}(x) B_j \adiag (x)- \sum^{\sigma}_{k=1} f_k (x) \adiag (x)B_k \\
& = & A(x) + \displaystyle\sum_{i=1}^{\sigma} f_i (x)[B_i,\adiag (x)].
\end{array}$$
Similarly, we have
$$
 P' (x)P^{-1}(x)= \left( \displaystyle\sum_{i=1}^{\sigma} f_i' (x) B_i\right)\left(\id_{\mathrm{n}} - \displaystyle\sum_{j=1}^{\sigma} f_j(x) B_j\right)=\displaystyle\sum_{i=1}^{\sigma} f_i' (x)B_i.$$
This yields the desired result.
\end{proof}

We have seen in Lemma \ref{diagsub} that $\mathrm{Lie_{alg}}\left(\asub ;\mathbf{k}\right)$ is an ideal in $\hlie (\mathbf{k})$. 
In particular, for all ${B(x)\in \mathrm{Lie_{alg}}\left(\asub ;\mathbf{k}\right)}$,
the bracket $[B(x),\adiag (x)]$ is in $\mathrm{Lie_{alg}}\left(\asub ;\mathbf{k}\right)$. 
This implies that the  $\mathbf{k}$-linear map $\Psi:=[\bullet,\adiag (x)]$, which is the adjoint action of $\mathrm{Lie_{alg}}\left(\adiag ;\mathbf{k}\right)$ on $\mathrm{Lie_{alg}}\left(\asub ;\mathbf{k}\right)$, is well defined:

$$\begin{array}{cccc}
\Psi: & \mathrm{Lie_{alg}}\left(\asub ;\mathbf{k}\right) & \longrightarrow & \mathrm{Lie_{alg}}\left(\asub ;\mathbf{k}\right) \\ 
 & B (x)&\longmapsto  & [B(x),\adiag (x)].
\end{array}$$

The following lemma will be necessary in $\S \ref{sec34}$. Note that the proof of the lemma gives a complete description of a finite set containing the eigenvalues of $\Psi$.
\pagebreak[3]
\begin{lemma}\label{lem3}
The eigenvalues of the linear map $\Psi$ belong to $\mathbf{k}$. 
\\
Furthermore, there exists a basis of \emph{constant} matrices, such that the matrix of the linear map $\Psi$ in this basis is block-diagonal, with blocks that are upper-triangular matrices with only one eigenvalue.
\end{lemma}

\begin{proof}
Let ${M_1,\ldots, M_\delta \in \CM_{\mathrm{n}}\left(\CC\right)}$ be a basis of $\mathrm{Lie_{alg}}\left(\adiag\right)$, which is abelian. We may write $\adiag (x)=\displaystyle\sum_{i=1}^{\delta} g_{i}(x) M_{i}$ with $g_{i}(x)\in \mathbf{k}$. 
Let $\Psi_i:=[\bullet,M_i]$ denote the adjoint action of $M_i$ on ${Lie_{alg}}\left(\asub\right)$.
As the matrices $M_i$ commute pairwise, the Jacobi identity on Lie brackets implies that the $\Psi_i$ also commute pairwise. 
The $\Psi_i$ have coefficients in the algebraically closed field $\CC$ and commute pairwise, 
and therefore they are simultaneously triangularizable 
in a basis $(C_j)$ of ${Lie_{alg}}\left(\asub\right)$. By construction, the $C_j$ are constant matrices. Each $C_j$ lies in a characteristic space of $\Psi_i$
associated with an eigenvalue $\lambda_{i,j}$. 
We define $\lambda_j(x):=\sum_{i=1}^{\delta} g_{i}(x)  \lambda_{i,j}$. As $\Psi=\sum_{i=1}^{\delta} g_{i}(x) \Psi_i$,
we see that the $\lambda_j(x)\in \mathbf{k}$ are the eigenvalues of $\Psi$ and that the matrix of $\Psi$ is triangular in the basis $(C_j)$ of ${Lie_{alg}}\left(\asub  ;\mathbf{k}\right)$.
\end{proof}

\begin{remark} One may refine this proof to predict the eigenvalues of $\Psi$.
Let $\gamma_{1}(x),\dots,\gamma_{\omega}(x)\in \mathbf{k}$ be the eigenvalues of  $\adiag (x)$. The above reasoning shows the existence of $P_{1}\in \mathrm{GL}_{\mathrm{n}}\left(\CC \right)$, such that $P_{1} \adiag (x) P_{1}^{-1}=:\left(\begin{array}{c|c|c} L_{1}(x) & &0 \\ \hline
 &\ddots &\\\hline
0&&L_{\omega}(x)\end{array}\right)$, where for $1\leq i\leq \omega$, $L_{i}(x)$ is a  matrix in coefficients in $\mathbf{k}$, with only one eigenvalue $\gamma_{i}(x)$. \par 
In the proof of Lemma \ref{lem3}, we have proved the existence of a basis of constant matrices, such that the matrix of the linear map $\Psi$ in this basis is block-diagonal, with blocks that are upper-triangular matrices corresponding to convenient restriction of the linear maps $\Psi_{i,j}: X_{i,j}\mapsto X_{i,j}L_{i}(x)-L_{j}(x)X_{i,j}$. For $1\leq i,j\leq \omega$, the map $\Psi_{i,j}$ admits only one eigenvalue that is equal to $\gamma_{i}(x)-\gamma_{j}(x)\in \mathbf{k}$. Then, the eigenvalues of $\Psi$ are of the form $\{\gamma_{i}(x)-\gamma_{j}(x),1\leq i,j\leq \omega\}$. Now the diagonal blocks are symmetric powers of $A_{1,red}(x)$; the latter 
has an abelian associated Lie algebra and is triangular. It follows that the $\gamma_{i}(x)$ are linear combinations (with integer coefficients) of the eigenvalues of $A_{1,red}(x)$, so that the eigenvalues of $\Psi$ also are  linear combinations (with integer coefficients) of the eigenvalues of $A_{1,red}(x)$.
\end{remark}

\subsection{Decreasing the Dimension of $\mathrm{Lie}(A;\mathbf{k})$.}\label{sec34}
We refer to $\S \ref{sec22}$, $\S \ref{sec31}$ and $\S \ref{sec32}$ for the notations and definitions used in this subsection. 
The aim of this section is to find a gauge transformation $P(x)$ such that $Y' (x)= P(x)[A(x)]Y(x)$ is in reduced form. 
Thanks to Corollary \ref{coro1}, it is sufficient to compute  a gauge transformation 
	${P(x)\in\Big\{\id_{\mathrm{n}}+B(x), B(x)\in  \mathrm{Lie_{alg}}\left(\asub ;\mathbf{k}\right)\Big\}}$ 
such that,
for every gauge transformation $\widetilde{Q}(x)\in\Big\{\id_{\mathrm{n}}+B(x), B(x)\in  \mathrm{Lie_{alg}}\left(\asub ;\mathbf{k}\right)\Big\}$,
we have $\mathrm{Lie}(P[A] ;\mathbf{k}) \subseteq  \mathrm{Lie}\left(\widetilde{Q}[P[A]];\mathbf{k}\right)$. \par 
The $\mathbf{k}$-linear adjoint map 
$\Psi=[\bullet,A_{diag}]:  \mathrm{Lie_{alg}}\left(\asub ;\mathbf{k}\right) \rightarrow \mathrm{Lie_{alg}}\left(\asub ;\mathbf{k}\right)$
 has its eigenvalues $\lambda_1 (x),\dots, \lambda_{\kappa} (x)$ in $\mathbf{k}$ (see Lemma \ref{lem3}) and its minimal polynomial has the form
$$
\Pi_{\Psi}(X) = \displaystyle\prod_{i=1}^{\kappa}\left(X- \lambda_{i}(x) \right)^{m_i}, \quad 
\textrm{ with } m_i \in \mathbb{N}^{*}.
$$
For each eigenvalue $\lambda_i (x)$, we let $E_{\lambda_i }:=\ker\left( \left(\Psi-\lambda_{i}(x)\id_{\sigma} \right)^{m_i}\right)$
be the corresponding characteristic space. 
So we have the standard decomposition 
	$\mathrm{Lie_{alg}}\left(\asub ;\mathbf{k}\right)=\bigoplus_{i=1}^{\kappa} E_{\lambda_i}$.
Of course,  the $E_{\lambda_i}$ are $\Psi$-invariant subspaces. 
Now $\mathrm{Lie_{alg}}\left(\asub ;\mathbf{k}\right)$  is also a $\Psi$-invariant subspace of $\mathrm{Lie_{alg}}\left(\asub ;\mathbf{k}\right)$.
As   the $E_{\lambda_i}$ have each a basis formed of \emph{constant} matrices (Lemma  \ref{lem3}), Proposition \ref{propo2} implies that
 we thus have 
$$  \mathrm{Lie_{alg}}\left(\asub ;\mathbf{k}\right) = 
	\bigoplus_{i=1}^{\kappa}  \left( E_{\lambda_i} \bigcap \mathrm{Lie_{alg}}\left(\asub ;\mathbf{k}\right) \right).$$
In the reduction process, we may (and will) hence perform a reduction on each $E_{\lambda_i}$ separately.
So, without loss of generality, we now assume that $\Psi$ has one eigenvalue $\lambda(x)\in\mathbf{k}$ and
$\Pi_{\Psi}(X) = (X- \lambda(x))^m$, for some $m\in \mathbb{N}^{*}$.
\\

As above, we let $E_{\lambda} := \ker\left( \left(\Psi-\lambda(x)\id_{\sigma} \right)^{m}\right)$
and, for $i\in\{0,\ldots,m\}$, let $E_{\lambda}^{(i)} := \ker\left( \left(\Psi-\lambda(x)\id_{\sigma} \right)^{i}\right)$.
We  have the standard flag decomposition 
$E_{\lambda} = \bigoplus_{i=1}^m E_{\lambda}^{(i)}/E_{\lambda}^{(i-1)}$.
And, last, we recall that for $M (x) \in  E_{\lambda}^{(i)}/E_{\lambda}^{(i-1)}$, we have 

	\begin{equation}\label{eq4}
	 \Psi(M(x)) = \lambda(x) M(x) + \widetilde{M}(x),
	\quad \textrm{ with } \; \widetilde{M}(x)\in E_{\lambda}^{(i-1)}.
	\end{equation}

\subsubsection{Reduction on One Level of a Characteristic Space} \label{reduction-one-level}
Let us first pretend that we know a basis $C_1,\ldots, C_t$ of 
$E_{\lambda}^{(m)}/E_{\lambda}^{(m-1)}$ 
(formed of constant matrices $C_i$, this is possible due to lemma \ref{lem3})
such that $C_{s+1}, \ldots, C_t$ form a basis of $\mathfrak{g}(\mathbf{k}) \cap \left(E_{\lambda}^{(m)}/E_{\lambda}^{(m-1)}\right)$.
This means that $C_1,\ldots, C_s$ could be ``removed'' by a gauge transformation.
\\
We decompose $A(x)$ as $A(x) = \bar{A}(x) + \displaystyle \sum_{i=1}^{t} a_i(x) C_i$, where 
$\bar{A}(x)\in E_{\lambda}^{(m-1)}$.
Our gauge transformation matrix is of the form $P(x) = \id_{\mathrm{n}} + \displaystyle\sum_{i=1}^t f_i(x) C_i$ with $f_{i}(x)\in\mathbf{k}$. 
As $\Psi(C_i) = \lambda(x) C_i + \widetilde{C}_i$, 
with $\widetilde{C}_i \in E_{\lambda}^{(m-1)}$, we apply Proposition \ref{propo2} to obtain:
$$ P[A] = \bar{A}(x) + \displaystyle\sum_{i=1}^t f_i(x) \widetilde{C}_i + \sum_{i=1}^t \left( a_i(x) + \lambda(x) f_i(x) - f_i'(x)\right) C_i. $$
We see that, in order to achieve reduction in $E_{\lambda}^{(m)}/E_{\lambda}^{(m-1)}$, we should have 
	$$ f_i'(x) = \lambda(x) f_i(x) + a_i(x) \; \textrm{ for all } \; i\in \{1,\ldots,s\}.$$
In other words, the differential equation $y'(x) = \lambda(x) y(x) + a_i(x)$ should have a rational solution for each $ i\in \{1,\ldots,s\}$.
\\\par 
In practice, we do not know the $C_i$ nor the $a_i(x)$ so we now show how to compute them. 
Let $B_1,\ldots, B_t$ denote a basis of 
$E_{\lambda}^{(m)}/E_{\lambda}^{(m-1)}$, 
formed of constant matrices.
We will find candidates for the $C_i$ by computing which combinations of the $B_i$ 
may be ``removed'' from $A(x)$ by a gauge transformation as above.
We decompose $A(x)$ as $A(x) = \bar{A}(x) + \displaystyle\sum_{i=1}^t b_i(x) B_i$.
There exist (yet unknown) constants $c_{i,j}$ such that $B_i = \displaystyle\sum_{j=1}^t c_{i,j} C_j$, so that: 
$$ A (x)= \bar{A} (x)+ \displaystyle\sum_{i=1}^t b_i(x) \left( \sum_{j=1}^t c_{i,j} C_j\right) 
	= \bar{A} (x)+ \sum_{j=1}^t \left( \sum_{i=1}^t c_{i,j}  b_i(x) \right) C_j.$$
So, the calculation from the previous paragraph shows that there should exist $g_j(x) \in \mathbf{k}$ such that,
for $j \in \{1,\ldots, s\}$, $g_j'(x) = \lambda(x) g_j(x) + \displaystyle\sum_{i=1}^t c_{i,j}  b_i(x)$.
The way to find $s$, the $g_j(x)$ and the $c_{i,j}$ is given by Lemma \ref{exp-parametre}
(which is proved here for convenience but is well known to specialists).

\begin{lemma}\label{exp-parametre} 
Let $\lambda(x), b_1(x),\ldots, b_t(x)$ be elements of $\mathbf{k}$. 
The set of tuples $(g(x), c_1,\ldots, c_t)\in \mathbf{k} \times \CC^t$ such that $g'(x)=\lambda(x) g(x) + \displaystyle\sum_{i=1}^t c_i b_i(x)$
is a $\CC$-vector space. Moreover, one can effectively compute a basis of this vector space.
\end{lemma}
\begin{proof}
Let $L_{\bf{\underline{b}}}$ be the linear differential operator of order $t$ whose solution space is spanned by 
$b_1(x), \ldots, b_t(x)$.
  Let $L:=L_{\bf{\underline{b}}}\cdot\left(\frac{d}{dx} - \lambda(x)\right)$, where the product is the composition,
i.e. the usual product in the non-commutative Ore ring $\mathbf{k}[\frac{d}{dx}]$. One readily sees that a function $g(x)\in \mathbf{k}$
satisfies $L(g(x))=0$ if and only if $L_{\bf{\underline{b}}}(g'(x)-\lambda(x) g(x))=0$, i.e. if there exist constants $c_i\in \CC$
such that $g'(x)-\lambda(x) g(x) = \displaystyle\sum_{i=1}^t c_i b_i (x)$. 
Hence, the set of tuples $(g(x), c_1,\ldots, c_t)\in \mathbf{k} \times \CC^t$ such that $g'(x) = \lambda(x) g(x) +\displaystyle \sum_{i=1}^t c_i b_i(x)$ is isomorphic with the set of rational solutions $g(x)$ of $L$.
The latter is a vector space whose basis can be effectively computed, see $\S\ref{base-field}$.
\end{proof}

Lemma \ref{exp-parametre} allows us to compute easily, see $\S\ref{base-field}$, a dimension $s\in \mathbb{N}$ and a basis
$\left( (g_j(x), \underline{c}_{(\bullet,j)}) \right)_{j=1..s}$ of elements in $\mathbf{k}\times \CC^t$ 
such that the equation $y'(x)=\lambda(x) y(x) + \displaystyle\sum_{i=1}^t  c_{i,j} b_i(x)$ has a rational solution $y(x)=g_j(x)$.
The unknown functions $a_i(x)$ that we were looking for are thus given by $a_i(x) = \displaystyle\sum_{i=1}^t  c_{i,j} b_i(x)$.\\
Via the incomplete basis theorem, we construct a constant invertible matrix $\overline{Q}\in \mathrm{GL}_{\mathrm{t}}\left(\overline{\mathcal{C}} \right)$ whose first $s$ columns are the  $\underline{c}_{(\bullet,j)}$.  
We may view $\overline{Q}$ as the base change matrix from the basis $(B_j)_{j=1}^{t}$ of $E_{\lambda}^{(m)}/E_{\lambda}^{(m-1)}$
to a new basis $(C_j)_{j=1}^{t}$ of $E_{\lambda}^{(m)}/E_{\lambda}^{(m-1)}$. 
Let $\gamma_{i,j}$ denote the entries of $\overline{Q}^{-1}$. 

\begin{lemma} \label{reduction-partielle}
Let $s\in \mathbb{N}$, $(g_j(x))_{j=1,\ldots,s}$, and $(\gamma_{i,j})$ be computed as in the above paragraph.
For $i \in \{1,\ldots, t \}$, let $f_{i}(x):=\displaystyle\sum_{j=1}^s \gamma_{i,j} g_j(x)$.
Finally, let $P_{\lambda}^{(m)}: =\id_{\mathrm{n}} + \displaystyle\sum_{i=1}^t f_{i}(x) B_i$. 
Then $P_{\lambda}^{(m)}$ is a partial reduction matrix, in the sense that 
\begin{equation}\label{eq7}
\mathrm{Lie_{alg}}\left( P_{\lambda}^{(m)}[A] ;\mathbf{k}\right) \cap 
		\left( E_{\lambda}^{(m)}/E_{\lambda}^{(m-1)} \right)
	=  \mathfrak{g}(\mathbf{k}) \cap 
		\left( E_{\lambda}^{(m)}/E_{\lambda}^{(m-1)} \right) .
\end{equation}
Furthermore, for all $\widetilde{Q}(x):=\id_{\mathrm{n}}+\displaystyle\sum_{i=s+1}^t h_{i}(x) C_i$ with $h_{s+1}(x),\dots,h_{t}(x)\in \mathbf{k}$, we have   
$$\mathrm{Lie}(P_{\lambda}^{(m)}[A] ;\mathbf{k})= \mathrm{Lie}\left(\widetilde{Q}[P_{\lambda}^{(m)}[A]] ;\mathbf{k}\right).$$
\end{lemma}

\begin{proof}
We apply the first point of Proposition \ref{propo1} (because $G$ is connected, see the proof of Lemma \ref{lem2}) to deduce that  $\mathfrak{g}(\mathbf{k})\subseteq \mathrm{Lie_{alg}}\left( P_{\lambda}^{(m)}[A] ;\mathbf{k}\right)$.
Then, 
\begin{equation}\label{eq1} \mathfrak{g}(\mathbf{k})\cap 
		\left( E_{\lambda}^{(m)}/E_{\lambda}^{(m-1)} \right)
	\subseteq \mathrm{Lie_{alg}}\left( P_{\lambda}^{(m)}[A] ;\mathbf{k}\right)  \cap 
		\left( E_{\lambda}^{(m)}/E_{\lambda}^{(m-1)} \right) .\end{equation}
		
We want to prove the equality.
By construction, $C_1,\ldots, C_s$ vanish in the construction of $P_{\lambda}^{(m)}[A]$
so that $C_{s+1},\ldots, C_t$ now form a basis of $\mathrm{Lie_{alg}}\left( P_{\lambda}^{(m)}[A] ;\mathbf{k}\right) \cap 
		\left( E_{\lambda}^{(m)}/E_{\lambda}^{(m-1)} \right)$.
Due to Theorem \ref{theo1}, there exists  $\widetilde{R}(x)=\id_{\mathrm{n}}+\displaystyle\sum_{i=s+1}^t h_i(x) C_i +R(x)$, with $h_i(x)\in \mathbf{k}$, $R(x)\in E_{\lambda}^{(m-1)}$, such that 
\begin{equation}\label{eq2}
\mathfrak{g}(\mathbf{k})\cap 
		\left( E_{\lambda}^{(m)}/E_{\lambda}^{(m-1)} \right)
	= \mathrm{Lie_{alg}}\left( \widetilde{R}[P_{\lambda}^{(m)}[A]] ;\mathbf{k}\right)  \cap		\left( E_{\lambda}^{(m)}/E_{\lambda}^{(m-1)} \right) .\end{equation}
	
But by construction, we have the inclusion 
\begin{equation}\label{eq3}
\mathrm{Lie}(P_{\lambda}^{(m)}[A] ;\mathbf{k})\cap 
		\left( E_{\lambda}^{(m)}/E_{\lambda}^{(m-1)} \right) \subseteq  \mathrm{Lie}\left(\widetilde{R}[P_{\lambda}^{(m)}[A]];\mathbf{k}\right)\cap 
		\left( E_{\lambda}^{(m)}/E_{\lambda}^{(m-1)} \right).
\end{equation}
Combining (\ref{eq1}), (\ref{eq2}) and (\ref{eq3})  proves (\ref{eq7}). \par 
Let $\widetilde{Q}(x):=\id_{\mathrm{n}}+\displaystyle\sum_{i=s+1}^t h_{i}(x) C_i$ with $h_{s+1}(x),\dots,h_{t}(x)\in \mathbf{k}$. 
By construction, we have 
\begin{equation}\label{eq8}
\mathrm{Lie}(P_{\lambda}^{(m)}[A] ;\mathbf{k})\cap \left( E_{\lambda}^{(m)}/E_{\lambda}^{(m-1)} \right)= \mathrm{Lie}\left(\widetilde{Q}[P_{\lambda}^{(m)}[A]] ;\mathbf{k}\right)\cap \left( E_{\lambda}^{(m)}/E_{\lambda}^{(m-1)} \right).
\end{equation}
Let $\widetilde{C}_{j}:=\Psi(C_{j})-\lambda (x)C_{j}$. We use (\ref{eq4}) and the fact that $\Psi$ is $\mathbf{k}$-linear plus Proposition \ref{propo2} to deduce the existence of $\underline{A}(x)\in \mathrm{Lie}(P_{\lambda}^{(m)}[A] ;\mathbf{k})\cap \left( E_{\lambda}^{(m)}/E_{\lambda}^{(m-1)} \right)$ such that
\begin{equation}\label{eq9}
P_{\lambda}^{(m)}(x)[A(x)]-\widetilde{Q}(x)[P_{\lambda}^{(m)}(x)[A(x)]]=\underline{A}(x)+\displaystyle\sum_{i=s+1}^t h_{i}(x)\widetilde{C}_{i}.
\end{equation}
Let $j\in \{s+1,\dots,t\}$. We know that $C_{j}\in \mathrm{Lie}(P_{\lambda}^{(m)}[A] ;\mathbf{k})$. 
By definition, the matrix $\widetilde{C}_{j}=\Psi(C_{j})-\lambda (x)C_{j}$ belongs to $\mathrm{Lie}(P_{\lambda}^{(m)}[A] ;\mathbf{k})\cap E_{\lambda}^{(m-1)}$. Due to (\ref{eq8}), it also belongs to $\mathrm{Lie}\left(\widetilde{Q}[P_{\lambda}^{(m)}[A]] ;\mathbf{k}\right)\cap E_{\lambda}^{(m-1)}$. Then, $\displaystyle\sum_{i=s+1}^t h_{i}(x)\widetilde{C}_{j}$ belongs to $\mathrm{Lie}(P_{\lambda}^{(m)}[A] ;\mathbf{k})\cap E_{\lambda}^{(m-1)}$ and $\mathrm{Lie}\left(\widetilde{Q}[P_{\lambda}^{(m)}[A]] ;\mathbf{k}\right)\cap E_{\lambda}^{(m-1)}$. 
We combine this fact and (\ref{eq9}) to deduce
$$\mathrm{Lie}(P_{\lambda}^{(m)}[A] ;\mathbf{k})\cap E_{\lambda}^{(m-1)}= \mathrm{Lie}\left(\widetilde{Q}[P_{\lambda}^{(m)}[A]] ;\mathbf{k}\right)\cap E_{\lambda}^{(m-1)}.$$
Combining (\ref{eq8}) and this equality, we find the result.
\end{proof}

\subsubsection{The Full Reduction Procedure}
The reduction procedure now is easy to establish by iterating the above process.
By assumption, all variational equations of lower order are in reduced form and have an abelian associated Lie algebra.
\\
Choose an eigenvalue $\lambda(x) \in \textrm{Spec}(\Psi)$ of the adjoint map $\Psi=[\bullet,A_{diag}]$.
Let $E_{\lambda}:=E_{\lambda}^{(m)}$ be the corresponding characteristic space. Let $l:=m$. 
\\
Compute a constant basis $(B_i)_{i=1..t}$ of $E_{\lambda}^{(l)}/E_{\lambda}^{(l-1)}$ and 
compute the partial reduction matrix $P_{\lambda}^{(l)}: =\id_{\mathrm{n}} + \displaystyle\sum_{i=1}^t f_i(x) B_i$
as in Lemma \ref{reduction-partielle}. Perform the transformation $A(x):=P_{\lambda}^{(l)}(x)[A(x)]$, 
let $l:=l-1$ and iterate this paragraph until $l=0$.
\\
When all these successive steps are performed, let $P_{\lambda} (x):=\displaystyle\prod_{l=1}^m P_{\lambda}^{(l)}(x)$.
Note that, by construction, the matrices $P_{\lambda}^{(l)}(x)$ commute pairwise so the order does not matter in the latter product.

Now perform this for all eigenvalues $\lambda(x) \in \textrm{Spec}(\Psi)$. 
The resulting matrix  is a reduced form.

\begin{theorem}\label{theoreme-reduction}
Using the algorithm and notations of the above paragraph, let 
$$P (x):=\prod_{\lambda(x) \in \textrm{Spec}(\Psi)} P_{\lambda} (x)
	\quad \textrm{ and } \quad 
	A_{\textrm{p,red}}(x):=P (x)[A (x)].$$
Then the system $Y'(x)=A_{\textrm{p,red}}(x) Y(x)$ is in reduced form and $P(x)$ is the corresponding reduction matrix.
\end{theorem}

\begin{proof}
Define $\widetilde{\asub} (x)$ as the off-diagonal part of $A_{\textrm{p,red}} (x)$ as in the rest of this section.
Pick any matrix $H(x)\in \mathrm{Lie_{alg}}\left(\widetilde{\asub};\mathbf{k}\right)\cap \left(E_{\lambda}^{(l)}/E_{\lambda}^{(l-1)}\right)$ for some
$\lambda(x) \in \textrm{Spec}(\Psi)$, for some integer $l$. 
 Let $\widetilde{Q}(x):=\id_{\mathrm{n}}+ H(x)$.
Then, Lemma \ref{reduction-partielle} implies that we have the equality
$\mathrm{Lie}(A_{\textrm{p,red}} ;\mathbf{k}) =  \mathrm{Lie}\left(\widetilde{Q}[A_{\textrm{p,red}}];\mathbf{k}\right)$.
Now, Lemmas \ref{diagsub} and \ref{lem1} show that any matrix in $ \Big\{\id_{\mathrm{n}}+B(x), B(x)\in  \mathrm{Lie_{alg}}\left(\widetilde{\asub} ;\mathbf{k}\right)\Big\}$ is a product of matrices $\id_{\mathrm{n}}+H(x)$ of the above form.
It follows that, for every gauge transformation $\widetilde{Q}(x)$ in the set
	${\Big\{\id_{\mathrm{n}}+B(x), B(x)\in  \mathrm{Lie_{alg}}\left(\widetilde{\asub} ;\mathbf{k}\right)\Big\}}$, 
we have $\mathrm{Lie}(A_{\textrm{p,red}} ;\mathbf{k})=  \mathrm{Lie}\left(\widetilde{Q}[A_{\textrm{p,red}}];\mathbf{k}\right)$. So, Corollary \ref{coro1} shows that 
the system $Y'(x)=A_{\textrm{p,red}}(x) Y(x)$ is in reduced form and $P(x)$ is the corresponding reduction matrix.
\end{proof}

\pagebreak[3]
\section{Back to the Morales-Ramis-Sim\'o Integrability Criterion} \label{racourci-MRS}

\subsection{Reducing the First Variational Equation}\label{sec43}
Initially we assumed that the first variational equation had been put into reduced form 
and had an abelian associated Lie algebra. 
However, 
the procedure described in this paper can be also used to put the first variational equation
into reduced form, i.e. to apply effectively the original Morales-Ramis integrability criterion. 
This allows us to recover the reduction method established by two of the authors in \cite{ApWe12b}.

First, factor the first variational equation, i.e. compute an equivalent lower block-triangular form differential system.
(see e.g. \cite{PS03}).
Then, apply a reduction procedure to the irreducible blocks on the diagonal (for example
the one of Aparicio-Compoint-Weil from \cite{ApCoWe13a}). 
This will put these blocks in diagonal form (maybe after an algebraic extension);
otherwise we have an obstruction to integrability (Boucher-Weil criterion, see \cite{BoWe03a,MoR}). If the blocks
have dimension $1$ or $2$, then a faster method using a variant of the Kovacic algorithm is given in \cite{ApWe12b}.

Once this is done, the method of this paper allows us to reduce the lower triangular blocks, 
thus putting the first variational equation into reduced form.

\subsection{The Effective Morales-Ramis-Sim\'o Integrability Criterion}

The Morales-Ramis-Sim\'o  integrability criterion states that if one of the variational equations of a Hamiltonian system 
has a differential Galois
group whose Lie algebra is not abelian, then it is not (meromorphically) Liouville integrable. For $p\in \mathbb{N}^{*}$, let $Y'(x)=A_{p}(x)Y(x)$ be the variational equation of order $p$, let $G_{p}$ be the differential Galois group of $Y'(x)=A_{p}(x)Y(x)$ and let $\mathfrak{g}_{p}$ be the Lie algebra of $G_{p}$.\par 
As we have seen in $\S \ref{sec43}$, we may use the procedure of $\S \ref{sec3}$ to put the first variational equation $Y'(x)=A_{1}(x)Y(x)$ in reduced form. If $\mathfrak{g}_{1}$ is not abelian, which can be checked easily, then the original Morales-Ramis integrability criterion fails. Let $p \geq 2$ and assume that, for all $m\in \{1,\dots,p-1\}$, we know a gauge transformation matrix $P_m(x)$ such that 
$P_m(x)[A_m(x)]$ is in reduced form, i.e. $\mathrm{Lie_{alg}}(P_m[A_m])= \mathfrak{g}_m$.
We further assume that each $\mathfrak{g}_m$ is abelian. Then, see $\S \ref{subsection:variational equations}$, the $p^{th}$ variational equation is of the form  
$$Y'(x)=
 A_{p}(x)Y(x), \hbox{ where }A_{p}(x):=\left(\begin{array}{c|c}\sym^{p}\left(A_{1}(x)\right) & 0 \\\hline S_{p}(x) & A_{p-1}(x)\end{array}\right)
$$
and the matrix $S_{p}(x)$ has entries in $\mathbf{k}$.  
Let $Q(x) := \left(\begin{array}{c|c} \Sym^{p}(P_1(x)) & 0 \\\hline 0 & P_{p-1}(x) \\ \end{array}\right)$ and consider (see $\S \ref{sec25}$)
$$ A(x) := Q(x)[A_p (x)]  = \left(\begin{array}{c|c} \sym^p\left(A_{1,red}(x)\right) & 0 \\\hline S(x) & A_{p-1,red}(x) \\\end{array}\right).$$
Let $P(x)$ be the gauge transformation that we have computed in $\S \ref{sec34}$. Then,  $$ A_{p,red}(x) := P(x)[A (x)]=P(x)[Q(x)[A_{p} (x)]] $$
is in reduced form.
If $\mathfrak{g}_{p}$ is not abelian, which now can be easily checked, the Morales-Ramis-Sim\'o integrability criterion fails. If $\mathfrak{g}_{p}$ is abelian, we may iterate the same procedure in order to put $Y'(x)=
 A_{p+1}(x)Y(x)$ in reduced form. \\ 
To summarize, for any $p\geq  2$, we are able to put the successive variational equations  $$Y'(x)=A_{1}(x)Y(x),\dots,Y'(x)=A_{p}(x)Y(x)$$ in reduced form or prove that one of the $\mathfrak{g}_{i}$ is not abelian.

\subsection{A Simplified Reduction Procedure}
In view of the applications of this reduction procedure to the Morales-Ramis-Sim\'o integrability criterion,
we have the following shortcut. We refer to $\S \ref{sec2}$ and $\S \ref{sec3}$ for the notations used in this subsection.
The Morales-Ramis-Sim\'o integrability criterion implies that, if the Hamiltonian system is integrable, once our reduced form from Theorem \ref{theoreme-reduction}
is computed, $\mathfrak{g}_p$  should be abelian for all $p\in \mathbb{N}^{*}$. 
With Lemma~\ref{diagsub}, we find that this is equivalent to saying 
that the resulting adjoint map $\Psi_{\textrm{red}}=[\bullet,A_{diag}]$ should be the zero map
(because $\mathrm{Lie_{alg}}(\asub)$ is always abelian and $\mathrm{Lie_{alg}}(\adiag)$ is assumed to be abelian).
\\
So, when performing the reduction, any characteristic space $E_\lambda$ corresponding to a non-zero
eigenvalue $\lambda(x) \in \textrm{Spec}(\Psi)$ must vanish. Also, for $\lambda=0$, all $E_0^{(l)}$ (for $l>2$) must 
vanish too.
As a consequence, if one is only interested in finding an obstruction to integrability but not necessarily a reduced form,
the reduction step in $\S \ref{reduction-one-level}$ can be significantly simplified.

Indeed (we use the notations from $\S \ref{reduction-one-level}$), instead of the equation with parametrized right-hand side in Lemma \ref{exp-parametre},
it is enough to look for a rational solution $g_i(x)$ to each equation $y'(x) = \lambda(x) y(x) + b_i(x)$.
If any of these equations does not have a rational solution, then the adjoint map $\Psi_{\textrm{red}}$ 
of the reduced form will still have the non-zero eigenvalue $\lambda(x)$, hence yielding an obstruction to abelianity 
of the associated Lie algebra.
\\
Otherwise, the partial reduction matrix of Lemma \ref{reduction-partielle} is easier to compute: just let
$P_{\lambda}^{(m)}(x): =\id_{\mathrm{n}} + \displaystyle\sum_{i=1}^t g_i(x) B_i$,
compute $P_{\lambda}^{(m)}(x)[A(x)]$, compute a basis $(B_i)$ of the new  space $E_{\lambda}^{(m-1)}$
 and iterate this reduction
as in $\S \ref{reduction-one-level}$.  Do this for all non-zero eigenvalues of $\Psi$.
For the zero eigenvalue, proceed similarly for the $E_0^{(l)}$ (for all $l>2$). Note that since $\lambda=0$, the problem is slightly easier. Indeed, using the notations from $\S \ref{reduction-one-level}$, we only have to check whether every $b_i(x)$ admits a primitive $g_{i}(x)\in \mathbf{k}$. If any of the $b_{i}(x)$ does not admit a primitive in $\mathbf{k}$, we obtain an obstruction to abelianity.  Otherwise, the partial reduction matrix will be $P_{0}^{(l)}(x): =\id_{\mathrm{n}} + \displaystyle\sum_{i=1}^t g_i(x) B_i$. 
If at this stage the process has not stopped, the partially reduced matrix has an associated Lie algebra
which is abelian so the application of the Morales-Ramis-Sim\'o integrability criterion now makes it necessary to go to the next higher
variational equation. 
\\
We may even iterate the process to the next variational equation without finishing the reduction: 
the only assumption that was used in our algorithmic construction was that the Lie algebra associated to
the previous variational equation was abelian. However, this is not very satisfying and one should, at this last step,
compute the reduced form by applying Lemma~\ref{exp-parametre} until the final case $\lambda=0$ and $m=1$. 
Since $\lambda=0$, the computations here are slightly easier.


\section{An Example}
Consider the Hamiltonian with potential given by $H:=\frac12 p_1^2 + \frac12 p_2^2 + V$, where the potential is given 
by $$ V= \frac{q_2}{q_3^2}\left( 9 q_1^2+ q_2^2\right).$$
This potential appears at the end of the \cite{CaDuMaPr10a} where the authors present it as a case 
where their necessary conditions are all satisfied so that one could guess that the system might be integrable, but it is left as an open case.
The corresponding Hamiltonian system is
$$(X_H):\quad  \left\{\begin {array}{ccc}
	\dot{q_1} &= & p_1\\ 
	\dot{q_2} &= & p_2 \\ 
	\dot{p_1} &= & 3\,{\frac {q_2\, \left( 3\,{q_1}^{2}+{q_2}^{2} \right) }{{q_1}^{4}}}\\ 
	\dot{p_2} &= &  -3\,{\frac {3\,{q_1}^{2}+{q_2}^{2}}{{q_1}^{3}}}\end {array}
 \right.
$$
Using the method of Darboux points and homothetic solutions (see \cite{CaDuMaPr10a} or the papers by the same authors or the papers by Combot in the references), we find a (rather obvious) pencil of particular solutions 
$$q_1=\lambda x, q_2=\sqrt{-3} q_1= \sqrt{-3}\lambda x, p_1= \dot{q_1} = \lambda, p_2=\sqrt{-3}  p_1= \sqrt{-3} \lambda.$$
To simplify the expression of later results, we choose $\lambda=4\,i{3}^{3/4}\sqrt {{\frac {i}{{m}^{2}-1}}}$; 
the pencil now depends on a free parameter $m$ and our particular solutions are:
$$
q_1 = 4\,i{3}^{3/4}\sqrt {{\frac {i}{{m}^{2}-1}}}x   
, q_2 = 12\cdot{3}^{1/4}\sqrt {{\frac {i}{{m}^{2}-1}}}x
, p_1 = 4\,i{3}^{3/4}\sqrt {{\frac {i}{{m}^{2}-1}}}
,p_2 = 12\cdot{3}^{1/4}\sqrt {{\frac {i}{{m}^{2}-1}}}.
$$

\subsection{First Variational Equation}

The first variational equation has matrix $A_1$ with 
$$ A_1 =\left( \begin {array}{cccc} 0&0&1&0\\ \noalign{\medskip}0&0&0&1
\\ \noalign{\medskip}\frac{3}{8}Ê\,{\frac {{m}^{2}-1}{{x}^{2}}}
	&{\frac {-i\sqrt {3} \left( {m}^{2}-1 \right) }{8{x}^{2}}}&0&0
\\ \noalign{\medskip}{\frac {-i\sqrt {3} \left( {m}^{2}-1 \right) }{8{x}^{2}}}
 &-  \,{\frac {{m}^{2}-1}{8{x}^{2}}}&0&0\end {array} \right)
.
$$
Following $\S \ref{sec43}$, we find a reduction matrix for this first variational equation:
$$
P_1 = \left( \begin {array}{cccc} x&1&i\sqrt {3}&1\\ \noalign{\medskip}-ix
\sqrt {3}&-i\sqrt {3}&1&-i/3\sqrt {3}\\ \noalign{\medskip}1&0&{\frac {
i/2\sqrt {3} \left( m+1 \right) }{x}}&-1/2\,{\frac {m-
1}{x}}\\ \noalign{\medskip}-i\sqrt {3}&0&1/2\,{\frac {m+1}{x}}
&{\frac {i/6\sqrt {3} \left( m-1 \right) }{x}}\end {array} \right).
$$
So, the reduced form of the first variational equation is 
$$A_{1,red} = \left( \begin {array}{cccc} 0&0&0&0\\ \noalign{\medskip}0&0&0&0
\\ \noalign{\medskip}0&0&\frac12 \,{\frac {m+1}{x}}&0
\\ \noalign{\medskip}0&0&0&-\frac12\,{\frac {m-1}{x}}\end {array}
 \right). 
$$
The associated Lie algebra is one dimensional (and abelian).
We hence turn to the second variational equation.

\subsection{Second Variational Equation}
The matrix of the  second variational equation is
$$
A_2(x) =\left(\begin{array}{c|c} \sym^2\left(A_1(x)\right) & 0 \\\hline S_{2}(x) & A_1(x) \\ \end{array}\right).
$$
We start with the partial reduction matrix 
$$
Q_{2,1}(x) =\left(\begin{array}{c|c} \Sym^2\left(P_1(x)\right) & 0 \\\hline  0& P_1(x) \\ \end{array}\right), 
$$
to obtain 
$$ A(x):=Q_{2,1}(x)[A_2(x)] = 
	\left(\begin{array}{c|c} \sym^2\left(A_{1,red}(x)\right) & 0 \\\hline S_{2,1}(x) & A_{1,red}(x) \\ \end{array}\right)
$$
where
\begin{eqnarray*}
S_{2,1}(x) &=&  c_2
 \left( \begin {array}{cccccccccc} 0&0&0&0&0&0&0&\frac{1}{x^3}&\frac{1}{x^2}&\frac{1}{x}\\ 
 \noalign{\medskip}0&0&0&0&0&0&0&-\frac{1}{x^2}&-\frac{1}{x}&-1
\\ \noalign{\medskip}0&0&{\frac {1}{xm}}& {\frac {1}{m}}  &0&{\frac {1}{m{x}^{2
}}}&{\frac {1}{xm}}&{\frac {10/3\,i\sqrt {3}}{m{x}^{2}}}&{\frac {10/3
\,i\sqrt {3}}{xm}}&{\frac {10/3\,i\sqrt {3}}{m}}\\ \noalign{\medskip}0
&0&-{\frac {1}{m{x}^{2}}}&-{\frac {1}{xm}}&0&-{\frac {1}{m{x}^{3}}}&-{
\frac {1}{m{x}^{2}}}&{\frac {-10/3\,i\sqrt {3}}{m{x}^{3}}}&{\frac {-10
/3\,i\sqrt {3}}{m{x}^{2}}}&{\frac {-10/3\,i\sqrt {3}}{xm}}\end {array}
 \right)
\\
&& \textrm{with } c_2=\frac1{48} \left( 1+i \right)  \left( {m}^{2}-1 \right) ^{3/2}\sqrt {2}  \cdot 3^{\frac{1}{4}}.
\end{eqnarray*}
The off-diagonal Lie algebra $\mathfrak{h}_{sub}$ is generated by four matrices and calculation shows that it has dimension 10.
The matrix $\Psi$ of the adjoint action $[A_{\textrm{diag}},\bullet]$ on $\mathfrak{h}_{sub}$ has eigenvalues 
$[-3 f_2,-2 f_2,-f_2,0,f_2,2 f_2,3 f_2]$, where $f_2:=\frac{m+1}{2 x}$ (eigenvalue of $A_{1,red}$), and is diagonalizable. Applying our algorithm produces the reduced form
$$A_{2,red}(x)\left(\begin{array}{c|c} \sym^2\left(A_{1,red}(x)\right) & 0 \\\hline S_{2,red}(x) & A_{1,red}(x) \\ \end{array}\right)$$
where
$$S_{2,red}(x) = c_2  \left( \begin {array}{cccccccccc} 0&0&0&0&0&0&0&0&0&0
\\ \noalign{\medskip}0&0&0&0&0&0&0&0&-\frac{1}{x} &0\\ \noalign{\medskip}0
&0&{\frac {1}{xm}}&0&0&0&0&0&0&0\\ \noalign{\medskip}0&0&0&-{\frac {1}{xm}}&0&0&0&0&0&0\end {array} \right).$$
We remark that $\mathrm{Lie}(A_{2,red})$ is one-dimensional (because $ xA_{2,red}$ is a constant matrix) whereas its algebraic 
envelope
$\mathrm{Lie_{alg}}\left(A_{2,red} \right)$ 
is two dimensional. This follows from the fact that an algebraic Lie algebra contains both the semi-simple and nilpotent part of each of its elements. This can also be seen by solving the reduced system. This is now very easy and the Picard-Vessiot extension is 
$\mathbf{k}( x^{\frac{m+1}{2}} , \ln(x))$, which has transcendence degree two over $\mathbf{k}$. A simple calculation (or a look at the Picard-Vessiot extension) shows that 
$\mathrm{Lie}\left(A_{2,red} \right)$ is again abelian so we may proceed to the third variational equation.

\subsection{Third Variational Equation}
We do what we did for the second variational equation; we perform the first partial reduction on the diagonal to obtain
$$ A(x):=Q_{3,1}(x)[A_3(x)] = 
	\left(\begin{array}{c|c} \sym^3\left(A_{1,red}(x)\right) & 0 \\\hline S_{3,1}(x) & A_{3,red}(x) \\ \end{array}\right).
$$
The off-diagonal Lie algebra $\mathfrak{h}_{sub}$ now has dimension 33.
The matrix $\Psi$ of the adjoint action $[A_{\textrm{diag}},\bullet]$ on $\mathfrak{h}_{sub}$ is no longer diagonalizable.
Letting again $f_2:=\frac{m+1}{2 x}$, the  minimal polynomial $\Pi_{\Psi}(X)$ of $\Psi$ is 
$$
{X}^{2} 
 \left( X - f_2 \right)^2 \left( X +  f_2 \right)^2 
  \left( X -2 f_2 \right)^2 \left( X + 2 f_2 \right)^2 
  \left( X - 3 f_2 \right) \left( X + 3 f_2 \right)
   \left( X -  4 f_2 \right) \left( X +  4 f_2 \right).
$$
Our reduction procedure turns $S_{3,1}(x)$ into $S_{3,red}(x):=\frac{c_2}{x} M_3$, where $M_3$ is the matrix
$$ 
{\displaystyle \left( \begin {array}{cccccccccccccccccccc} 0&0&0&0&0&0&0&0&0&0&0&0&0
&0&0&0&0&0&0&0\\ \noalign{\medskip}0&0&0&0&0&0&0&0&-1&0&0&0&0&0&0&0&0&0
&0&0\\ \noalign{\medskip}0&0&\frac{2}{m}&0&0&0&0&0&0&0&0&0&0&0&0&0&0&0
&0&0\\ \noalign{\medskip}0&0&0&-\frac{2}{m}&0&0&0&0&0&0&0&0&0&0&0&0&0&0
&0&0\\ \noalign{\medskip}0&0&0&0&0&0&0&0&0&0&0&0&0&0&-1&0&0&0&0&0
\\ \noalign{\medskip}0&0&0&0&0&\frac{1}{m}&0&0&0&0&0&0&0&0&0&0&0&-2&0&0
\\ \noalign{\medskip}0&0&0&0&0&0&-\frac{1}{m}&0&0&0&0&0&0&0&0&0&0&0&-2&0
\\ \noalign{\medskip}0&0&0&0&0&0&0&\frac{2}{m}&0&0&0&0&0&0&0&0&0&0&0&0
\\ \noalign{\medskip}0&0&0&0&0&0&0&0&0&0&0&0&0&0&0&0&0&0&0&0
\\ \noalign{\medskip}0&0&0&0&0&0&0&0&0&-\frac{2}{m}&0&0&0&0&0&0&0&0&0&0
\end {array} \right)}.
$$
The associated Lie algebra $Lie_{alg}(A_{3,red})$ is still two-dimensional and is still abelian.
Actually, the reduced system is easily solved and its Picard-Vessiot extension is the same as that of $VE_2$ so they  
still have the same (abelian) differential Galois group.

\section{Conclusion}
The reduction procedure established in this paper gives an effective version of the Morales-Ramis-Sim\'o criterion in the sense that it allows us to effectively test whether an $p$-th variational equation has an abelian Lie algebra. 
However, when the first $p-1$ variational equations have an abelian Lie algebra but the $p$-th does not, there is no known way to measure \emph{a priori} which $p$ would be needed. So, one may apply the reduction iteratively to higher and higher order but there is no criterion for determining when to stop. 
Also, when all variational equations have an abelian Lie algebra, the system could still be non-integrable (but one would see this on the variational equations along another particular solution).\\
This reduction procedure will also allow further study of how the dimensions of the Galois groups of the successive variational equations evolve, both in integrable and non-integrable situations.\\

The reduced form may also be combined with the methods of \cite{ApBaSiWe11a,Si14b} for finding Taylor expansions of first integrals. Once the system is in reduced form, the results of \cite{ApCoWe13a} show that the Taylor expansions of a first integral, along the particular solution $\Gamma$, have constant coefficients. So, once the system is in reduced form, the (eventual) first integrals are easily found. In that sense, our reduced forms appear as pre-normal forms along $\Gamma$.  Pushing the reduction further to develop a normal form theory would be a natural development.
\\

The concepts of variational equations are \emph{mutatis mutandis} the same for general (non-Hamiltonian) dynamical systems (see e.g. \cite{Ca09a} or \cite{CaWe15a}, where several notions of variational equations are compared). The notion of Liouville integrability may be generalized to these contexts by Bogoyavlenskij integrability: the notion of involution of first integrals is replaced by the (equivalent) notion of commuting vector fields, see \cite{AyZu10a,BaCu05a,Bo96b,CuBa97a}. The Morales-Ramis-Sim\'o theory is generalized in (\cite{AyZu10a,Ca09a}) to any kind of ordinary differential systems. 
The reader may remark that, in this paper,  we essentially never use  the symplectic structure of the Hamiltonian system from which we started. Hence, the reduction methods that we  developed in the (symplectic) Morales-Ramis-Sim\'o context extends naturally to any Bogoyavlenskij integrable differential system.\\

Our reduction procedure is interesting in its own right because it applies to  
other kinds of "solvable" situations that can be found in the context of differential Galois theories.
Indeed, consider  a differential system of the form $Y'=A(x)Y$ where $A(x)$ has the form
	$$ A(x) =    \left(\begin{array}{c|c} 
	      A_1(x) & 0 \\\hline
	      S(x) & A_2(x) \\
	   \end{array}\right).$$
Assume that the block-diagonal part $\left(\begin{array}{c|c} 
	      A_1(x) & 0 \\\hline
	     0 & A_2(x) \\
	   \end{array}\right)$ is in reduced form and has an abelian associated Lie algebra.
Our reduction procedure readily extends to this (slightly more general) situation and puts the system into reduced form. 
In particular, it may be viewed as a way to pre-simplify the solutions.
\\

Last, we mention the case of diagonals with a non-abelian Lie algebra. In \cite{CaWe15a}, Casale and Weil develop a similar reduction technique to a family of systems in the above form but where $\left(\begin{array}{c|c} 
	      A_1(x) & 0 \\\hline
	      0 & A_2(x) \\
	   \end{array}\right)$ has a non-abelian Lie algebra. Mixing these ideas and the ones developed in this work may provide a way toward a reduction method for general reducible linear differential systems.



\bibliographystyle{amsalpha} 
\bibliography{ADW_MRS_biblio}
\end{document}